%% file: article.tex
\documentclass[a4paper, 11pt]{amsart}
\usepackage{amsthm}
\usepackage[]{amsmath}
\usepackage{amssymb}
\usepackage{enumerate}
\usepackage{tabularx}
\usepackage[]{color}
\usepackage[left=3cm, right=3cm]{geometry}
\usepackage[colorlinks]{hyperref}
\usepackage{tikz}
\usepackage{multirow}
\usepackage{subcaption}
\usepackage{algorithm}
\usepackage{algorithmic}
\usepackage{multirow}

\newcommand{\bm}[1]{\boldsymbol{#1}}
\newcommand{\bmr}[1]{\bm{\mr{#1}}}
\newcommand{\lj}{[ \hspace{-2pt} [}
\newcommand{\rj}{] \hspace{-2pt} ]}
\newcommand{\mb}[1]{\mathbb{#1}}
\newcommand{\mc}[1]{\mathcal{#1}}
\newcommand{\mr}[1]{\mathrm{#1}}
\newcommand{\jump}[1]{\lj #1 \rj}

\newcommand{\wt}[1]{ \widetilde{ #1}}
\newcommand{\wh}[1]{ \widehat{ #1}}
\newcommand{\tr}[1]{\ifmmode \mathrm{tr}\left( #1 \right) \else 
\text{tr} \left( #1 \right) \fi }

\newcommand\enorm[1]{|\!|\!| #1 |\!|\!|}

\newcommand{\npar}[1]{ \frac{\partial{#1}}{\partial \un} } 

\renewcommand{\d}[1]{\mathrm d \boldsymbol{#1}}
\newcommand\pnorm[1]{\| #1 \|_{\bm{\mathrm{p}}}}
\newcommand\unorm[1]{\| #1 \|_{\bm{\mathrm{u}}}}

\newcommand\comment[1]{}

\def\MTh{\mc{T}_h}
\def\MEh{\mc{E}_h}
\def\un{\bm{\mr{n}}}
\def\curl{\ifmmode \mathrm{curl} \else \text{curl}\fi}
\def\div{\ifmmode \mathrm{div} \else \text{div}\fi}
\def\RT{\ifmmode \mathrm{\bf RT} \else \text{\bf RT} \fi}
\def\ui{\bm{\mr i}}

\newtheorem{theorem}{Theorem}
\newtheorem{lemma}{Lemma}

\newtheorem{remark}{Remark}

\title[DLS Method for Helmholtz]{A Discontinuous Least Squares Finite
  Element Method for Helmholtz Equations}

\author[R. Li]{Ruo Li} \address{CAPT, LMAM and School of Mathematical
  Sciences, Peking University, Beijing 100871, P.R. China}
\email{rli@math.pku.edu.cn}

\author[Q.-C. Liu]{Qicheng Liu} \address{School of Mathematical
  Sciences, Peking University, Beijing 100871, P.R. China}
\email{qcliu@pku.edu.cn}

\author[F.-Y. Yang]{Fanyi Yang} \address{School of Mathematical
  Sciences, Peking University, Beijing 100871, P.R. China}
\email{yangfanyi@pku.edu.cn}

\begin{document}

\maketitle

\input{abstract.tex}
\input{introduction.tex}
\input{preliminaries.tex}
\input{lsmethod.tex}
\input{numericalresults.tex}
\input{conclusion.tex}

\bibliographystyle{amsplain}
\bibliography{../ref}

\end{document}

%% file: abstract.tex
\begin{abstract}
  We propose a discontinuous least squares finite element method for
  solving the Helmholtz equation. The method is based on the $L^2$
  norm least squares functional with the weak imposition of the
  continuity across the interior faces as well as the boundary
  conditions. We minimize the functional over the discontinuous
  polynomial spaces to seek numerical solutions. The wavenumber
  explicit error estimates to our method are established. The optimal
  convergence rate in the energy norm with respect to a fixed
  wavenumber is attained. The least squares functional can naturally
  serve as {\it a posteriori} estimator in the $h$-adaptive
  procedure. It is convenient to implement the code due to the usage
  of discontinuous elements. Numerical results in two and three
  dimensions are presented to verify the error estimates.

  \noindent \textbf{keywords}: Helmholtz equation, Least squares
  method, Discontinuous elements, Error estimates.
\end{abstract}


%% file: introduction.tex
\section{Introduction}
\label{sec_introduction}
The Helmholtz equation is applicable in many physical applications
involving time-harmonic wave propagation phenomena such as linear
acoustics, elastodynamics and electrodynamics
\cite{Thompson1995Galerkin, Hu2020novel,Farhat2003discontinuous, 
Nguyen2015hybridizable}. These important applications drive people to
construct numerical methods to the Helmholtz equation 
\cite{Nguyen2015hybridizable}. The Helmholtz operator is indefinite 
with large wave numbers, which brings difficulties in developing 
efficient numerical schemes and establishing stability estimates 
\cite{Feng2009discontinuous}. It is well known that the quality of 
discrete numerical solutions to the Helmholtz equation dramatically 
depends on the wavenumber $k$, known as the pollution effect 
\cite{Babuska2000pollution}. In spite of such difficulties, there 
have been plenty of researches on numerical methods to this problem, 
such as finite element methods, spectral methods and discontinuous 
Galerkin methods. 

The finite element methods are widely used for solving the Helmholtz
equation. A common choice is to use the standard conforming elements
to approximate the solution.  We refer to \cite{Ihlenburg1995finite,
Ihlenburg1997finite, Melenk2011wavenumber} for more details of these
conforming methods.  Compared with conforming finite element methods,
discontinuous Galerkin methods (DGMs) have several attractive features
on the mesh structure \cite{Chen2013hybrid}. Without the continuity
condition across interelement boundaries, the DGMs can be easily
applied on the general mesh structure which may include different
shapes of the elements and hanging points, and can allow the
polynomial degrees be different from element to element. Thus, the
DGMs have been applied in the numerical simulation of Helmholtz
equation. We refer to \cite{Feng2009discontinuous,
Feng2011discontinuous, Congreve2019robust, Hoppe2013convergence,
Feng2013absolutely, Chen2013hybrid, Hiptmair2011plane} and the
references therein for some typical DGMs. 

The least squares finite element method (LSFEM) is a general
numerical method, which is based on the minimization of a quadratic
functional, and we refer to \cite{Bochev1998review} for an overview to
this method. The resulting system arising from most of the
above Galerkin finite element methods are indefinite with a large
wavenumber, while the LSFEM can always provide a positive definite
linear system \cite{Chang1990least, Chen2017first}. Considering this
attractive property, LSFEM has been applied to numerically solve the
Helmholtz equation \cite{Lee2000first, Chen2017first,
Thompson1995Galerkin, Chang1990least, Hu2020novel, Monk1999least}.

In this paper, we propose a discontinuous least squares finite element
method. We introduce an $L^2$ norm least squares functional involving
the proper penalty terms which weakly enforce the continuity across
the interior faces as well as the boundary conditions. The numerical
solution is sought by minimizing the functional over piecewise
polynomial spaces. Such similar ideas have been applied to many
problems, see \cite{Bensow2005div, Bensow2005discontinuous,
Bochev2012locally, li2019least}. With discontinuous elements, the
proposed method is easily implemented and has great flexibility on the
mesh structure. The discretized system is still shown to be symmetric
positive definite. Generally, the advantages of DGM and LSFEM are
combined in this method. 

In finite element methods, the pollution effect lies in the constant
$C$ of the error estimate as the wavenumber increases
\cite{Babuska2000pollution, Melenk2011wavenumber}. For the proposed
method, we establish the wavenumber explicit error estimate. Our
method is shown to be stable without any assumption on
the mesh size. We prove that with respect to the fixed wavenumber,
our method has an optimal convergence rate in the energy norm and a
sub-optimal convergence rate in the $L^2$ norm. We observe that the
constants in the energy error and $L^2$ error are of the order of
$O(k^2)$ and $O(k)$, respectively. Our theoretical estimates are 
verified by some numerical experiments in two and three dimensions. 
We also include an example to numerically explore the pollution 
effect as the wavenumber $k$ increases. We note that the least squares
functional naturally provides an {\it a posteriori} indicator, and
from this we present an $h$-adaptive algorithm and test its accuracy by
solving a low-regularity problem.

The rest of this paper is organized as follows. In Section
\ref{sec_preliminaries}, we introduce the notation and define the
first-order system to the Helmholtz equation. The $k$-explicit
stability result of the Helmholtz equation is also recalled in this
section. In Section \ref{sec_method}, we define the least squares
functional and propose our least squares method. The analysis of
errors is also given in this section. In Section
\ref{sec_numericalresults}, we conduct a series of 2D and 3D numerical
tests to demonstrate the accuracy of the proposed method.


%% file: preliminaries.tex
\section{preliminaries}
\label{sec_preliminaries}
Let $\Omega_1 \in \mb{R}^d$ be an open, bounded, strictly star-shaped 
polygonal (polyhedral) domain, where $d = 2$ or $3$. $D \subset
\Omega_1$ is a star-shaped domain, which represents a scatterer. 
We define $\Omega = \Omega_1 \backslash D$ and $\Gamma_R = \partial 
\Omega_1$, $\Gamma_D = \partial D$. In this paper, we concern the
following Helmholtz equation: seek $u$ such that
\begin{equation}
  \begin{aligned}
    -\Delta u - k^2 u &= f &&\text{in } \Omega, \\
    u &= g_0 &&\text{on } \Gamma_D, \\
    \npar{u} + \ui k u &= g && \text{on } \Gamma_R, \\
  \end{aligned}
    \label{eq_H3}
\end{equation}
where $k>0$ is the wavenumber, $\ui = \sqrt{-1}$ is the imaginary
unit, and $\un$ denotes the unit outward normal to $\Omega$. The Robin
boundary condition of \eqref{eq_H3} is known as the first order
absorbing boundary condition \cite{Engquist1979radiation}. We allow
the case $D = \varnothing$.

We denote by $\MTh$ a shape-regular triangulation over the
domain $\Omega$. Let $\MEh^i$ be the collection of all $d - 1$
dimensional interior faces with respect to the partition $\MTh$,
$\MEh^D$ be the collection of all $d - 1$ dimensional faces that are
on the boundary $\Gamma_D$ and $\MEh^R$ be the collection of all $d-1$
dimensional faces that are on the boundary $\Gamma_R$. We then set
$\MEh:= \MEh^i \cup \MEh^D \cup \MEh^R$. For any element $K \in \MTh$
and any face $e \in \MEh$, we let $h_K$ and $h_e$ be their diameters,
respectively, and we denote by $h := \max_{K \in \MTh} h_K$ the mesh
size of $\MTh$. Then the shape regularity of $\MTh$ is in the sense of
that: there exists a constant $C>0$ such that 
\begin{displaymath}
  \frac{h_K}{\rho_K} \leq C,
\end{displaymath}
for any element $K \in \MTh$, and $\rho_K$ denotes the diameter of the
largest disk (ball) inscribed in $K$.

Next, we introduce the following trace operators which are commonly
used in the DG framework. 
For the scalar-valued piecewise smooth function $v$ and the
vector-valued piecewise smooth function $\bm{v}$, we 
define the jumps of $v$ and $\bm{v}$ on the interior face 
$e = \partial K^+ \cap \partial K^-$ as
\begin{displaymath}
  \begin{aligned}
    \jump{v} &:= v|_{K^+} \un^+ + v|_{K^-} \un^-, 
    \text{ for scalar-valued } v, \\
    \jump{\un \cdot \bm{v}} &:= \un^+ \cdot \bm{v}|_{K^+} + \un^-
    \cdot \bm{v}|_{K^-}, \text{ for vector-valued } \bm{v}, \\
  \end{aligned}
\end{displaymath}
where $\un^+$ and $\un^-$ are the unit outward normal to $e$ of $K^+$
and $K^-$, respectively. For the boundary face $e \in \MEh^D \cup
\MEh^R$, we set 
\begin{displaymath}
  \begin{aligned}
    \jump{v} &:= v \un, \text{ for scalar-valued } v, \\
    \jump{\un \cdot \bm{v}} &:= \un \cdot \bm{v}, 
    \text{ for vector-valued } \bm{v}, \\
  \end{aligned}
\end{displaymath}
where $\un$ is the unit outward normal to $e$. 

Given the bounded domain $Q$, we follow the standard notations
$L^2(Q)$, $L^2(Q)^d$, $H^r(Q)$ and $H^r(Q)^d$ to represent
the {\it complex-valued} Sobolev spaces with the regular exponent $r
\geq 0$.  The $L^2$ inner products to these spaces are defined as
\begin{displaymath}
  \begin{aligned}
    (u,v)_{L^2(Q)} &:= \int_Q u \ \overline{v} \ \d{x}, \text{ for
    scalar-valued Sobolev spaces}, \\
    (\bm{u},\bm{v})_{L^2(Q)} &:= \int_Q \bm{u} \cdot \overline{\bm{v}}
    \ \d{x},\text{ for vector-valued Sobolev spaces},
  \end{aligned}
\end{displaymath}
and the corresponding semi-norms and norms are induced from the $L^2$
inner products. Further, we denote by $H^r_D(\Omega)$ the space of
functions in $H^r(\Omega)$ with vanishing trace on $\Gamma_D$,
\begin{displaymath}
  H^r_D(\Omega) := \left\{ v \in H^r(\Omega) \ | \ v = 0, \text{ on }
  \Gamma_D \right\}.
\end{displaymath}
Besides, the following space will be used in our analysis,
\begin{displaymath}
  H(\div, \Omega) := \left\{ \bm{v} \in L^2(\Omega)^d \ | \ \nabla \cdot
  \bm{v} \in L^2(\Omega) \right\},
\end{displaymath}
with the norm
\begin{displaymath}
  \| \bm{v} \|_{H(\div, \Omega)}^2 := \| \bm{v} \|_{L^2(\Omega)}^2 +
  \| \nabla \cdot \bm{v} \|_{L^2(\Omega)}^2.
\end{displaymath}
For the partition $\MTh$, we will use the notations and the
definitions for the broken Sobolev space $L^2(\MTh)$, $L^2(\MTh)^d$,
$H^r(\MTh)$ and $H^r(\MTh)^d$ with the exponent $r \geq 0$ and their
associated inner products and norms \cite{arnold2002unified}. We note 
that the capital $C$ with or without subscripts are generic positive 
constants, which are possibly different from line to line, but are 
independent of the mesh size $h$ and the wavenumber $k$.

Under the assumptions of the domain $\Omega$, the following
$k$-explicit stability result of the Helmholtz equation holds true,
which is critical in our error estimates:
\begin{theorem}
Suppose $g_0 = 0$, and $\Omega_1$ is a strictly star-shaped domain and $D
\subset \Omega_1$ is a star-shaped domain. Let $k_0$ be an arbitrary 
strictly positive number. Then there is a constant $C > 0$ such that 
for any $f \in L^2(\Omega)$ , $g \in L^2(\Gamma_R)$, and $k
\geq k_0$, the Helmholtz equation \eqref{eq_H3} has a unique 
solution $u \in H_D^1(\Omega)$ satisfying
\begin{equation}
  k\| u \|_{L^2(\Omega)} + \| \nabla u \|_{L^2(\Omega)} \leq C \left(
  \| f \|_{L^2(\Omega)} + \| g \|_{L^2(\Gamma_R)} \right).
  \label{eq_stability}
\end{equation}
\label{th_stability}
\end{theorem}
We refer to \cite[Section 3.4]{Hetmaniuk2007stability} for details of 
this result.

In this paper, we propose a least squares finite element method for
the Helmholtz equation \eqref{eq_H3} based on the discontinuous
approximation. We begin by introducing an auxiliary 
variable $\bm{p} = \frac{1}{k} \nabla u$ to recast the Helmholtz 
equation \eqref{eq_H3} into a first-order system, 
\begin{equation}
  \begin{aligned}
    -\nabla \cdot \bm{p} - k u = \wt{f}, &\quad \text{in }
    \Omega, \\
    \nabla u - k \bm{p} = \bm{0}, &\quad \text{in }
    \Omega, \\
    u = g_0, &\quad \text{on } \Gamma_D, \\
    \un \cdot \bm{p} + \ui u = \wt{g}, &\quad \text{on } \Gamma_R,
  \end{aligned}
  \label{eq_firstHelmholtz}
\end{equation}
where $\wt{f} = \frac{1}{k} f$ and $\wt{g} = \frac{1}{k} g$. The
variable $u$ and $\bm{p}$ give the electric field and the magnetic
field, respectively. Rewriting the problem into a
first-order system is the fundamental idea in the modern least squares
finite element method \cite{Bochev1998review, Lee2000first,
Chen2017first}, and our discontinuous least squares
method is then based on the system \eqref{eq_firstHelmholtz}.

%% file: lsmethod.tex
\section{Discontinuous Least Squares Method for Helmholtz Equation}
\label{sec_method}
Aiming to construct a discontinuous least squares finite element
method for the system \eqref{eq_firstHelmholtz}, we first define a
least squares functional based on \eqref{eq_firstHelmholtz}, which
reads
\begin{equation}
  \begin{aligned}
    J_h(u, \bm{p}) &:= \sum_{K \in \MTh} \left( \| \nabla \cdot \bm{p}
    + ku + \wt{f} \|^2_{L^2(K)} + \| \nabla u - k\bm{p} \|^2_{L^2(K)}
    \right) \\
    &+ \sum_{e \in \MEh^i} \frac{1}{h_e} \left( \| \jump{u}
    \|^2_{L^2(e)} + \| \jump{\un \cdot \bm{p}} \|^2_{L^2(e)} \right) \\
    &+ \sum_{e \in \MEh^D} \frac{1}{h_e} \| u - g_0 \|^2_{L^2(e)}
    + \sum_{e \in \MEh^R} \frac{1}{h_e}
    \| \un \cdot \bm{p} + \ui u - \wt{g} \|^2_{L^2(e)},
  \end{aligned}
  \label{eq_functional}
\end{equation}
The terms in \eqref{eq_functional} defined on $\MEh^i$, $\MEh^D$ and 
$\MEh^R$ weakly impose the continuity condition and the boundary
condition, respectively.

Then we introduce two approximation spaces $\bmr{V}_h^m$ and 
$\bmr{\Sigma}_h^m$ for the variables $u$ and $\bm{p}$, respectively:  
\begin{displaymath}                                                   
  \bmr{V}_h^m := V_h^m, \qquad \bmr{\Sigma}_h^m := (V_h^m)^d,         
\end{displaymath}                                                     
where $V_h^m$ is the {\it complex-valued} piecewise polynomial space,
\begin{displaymath}
  V_h^m := \left\{ v_h \in L^2(\Omega) \ | \ v_h|_K \in \mb{P}_m(K),
    \  \forall K \in \MTh \right\}.                                                   
\end{displaymath}                                                     
One can write any function $v_h \in \bmr{V}_h^m$ and any function
$\bm{q}_h \in \bmr{\Sigma}_h^m$ as                                             
\begin{equation*}                                                
  v_h = \sum_l v_l \varphi_l, \quad \bm{q}_h = \sum_l q_l \bm{\psi}_l,    
\end{equation*}                                                       
where $\{ \varphi_l \}$ is a basis of the standard real-valued scalar
piecewise polynomial space, and $\{ \bm{\psi_l} \}$ is a basis of the
standard real-valued vector piecewise polynomial space, and $\{v_l\}$
and $\{ q_l \}$ are both complex combination coefficients. Apparently,
the functions in both spaces $\bmr{V}_h^m$ and $\bmr{\Sigma}_h^m$ may
be discontinuous across interior faces.  In this paper, we seek the
numerical solution $(u_h, \bm{p}_h) \in \bmr{V}_h^m \times
\bmr{\Sigma}_h^m$ by minimizing the functional \eqref{eq_functional}
over the space $\bmr{V}_h^m \times \bmr{\Sigma}_h^m$, which takes the
form:
\begin{equation}
  (u_h, \bm{p}_h) = \mathop{\arg \min}_{(v_h, \bm{q}_h) \in
  \bmr{V}_h^m \times \bmr{\Sigma}_h^m} J_h(v_h, \bm{q}_h).
  \label{eq_minJ}
\end{equation}
To solve the minimization problem \eqref{eq_minJ}, we can write the
corresponding Euler-Lagrange equation, which reads:
{\it find $(u_h, \bm{p}_h) \in \bmr{V}_h^m \times \bmr{\Sigma}_h^m$
such that}
\begin{equation}
  a_h(u_h, \bm{p}_h; v_h, \bm{q}_h) = l_h(v_h,
  \bm{q}_h), \qquad \forall (v_h, \bm{q}_h) \in \bmr{V}_h^m \times
  \bmr{\Sigma}_h^m,
  \label{eq_bilinear}
\end{equation}
where the bilinear form $a_h(\cdot; \cdot)$ and the linear form
$l_h(\cdot)$ are defined as
\begin{equation}
  \begin{aligned}
    a_h(u_h, \bm{p}_h; v_h, \bm{q}_h) &:= \sum_{K \in \MTh} \int_K
    (\nabla \cdot \bm{p}_h + ku_h) \ \overline{(\nabla \cdot \bm{q}_h +
    kv_h)} \d{x} \\
    &+ \sum_{K \in \MTh} \int_K (\nabla u_h - k\bm{p}_h) \cdot
    \overline{(\nabla v_h - k\bm{q}_h)} \d{x} \\
    &+ \sum_{e \in \MEh^i} \frac{1}{h_e} \left( \int_e \jump{u_h}
    \cdot \overline{\jump{v_h}} \d{s} + \int_e \jump{\un \cdot
    \bm{p}_h} \ \overline{\jump{\un \cdot \bm{q}_h}} \d{s} \right) \\
    &+ \sum_{e \in \MEh^D} \frac{1}{h_e} \int_e u_h \ \overline{v_h}
    \d{s} + \sum_{e \in \MEh^R} \frac{1}{h_e} \int_e (\un
    \cdot \bm{p}_h + \ui u_h) \ \overline{(\un \cdot \bm{q}_h + \ui
    v_h)} \d{s},
  \end{aligned}
  \label{eq_bilinearform}
\end{equation}
and
\begin{displaymath}
  \begin{aligned}
    l_h(v_h, \bm{q}_h) &:= \sum_{K \in \MTh} \int_K \wt{f} \ 
    \overline{\nabla \cdot \bm{q}_h} \d{x} +
    \sum_{K \in \MTh} \int_K f \  \overline{v_h}  \d{x} \\
    &+ \sum_{e \in \MEh^D} \frac{1}{h_e} \int_e g_0 \ \overline{v_h}
    \d{s} \\
    &+ \sum_{e \in \MEh^R} \frac{1}{h_e} \int_e \wt{g} \ 
    \overline{(\un \cdot q_h + \ui v_h)} \d{s}
  \end{aligned}
\end{displaymath}
Next, we will derive the error estimates to the problem
\eqref{eq_bilinear} and focus on how the error bounds depend on the
wavenumber $k$.
To do so, we first define two spaces $\bmr{V}_h$ and $\bmr{\Sigma}_h$
for variables $u$ and $\bm{p}$, respectively, as
\begin{displaymath}                                                   
  \bmr{V}_h := \bmr{V}_h^m + H_D^1(\Omega), \qquad \bmr{\Sigma}_h :=  
  \bmr{\Sigma}_h^m + H(\div, \Omega),
\end{displaymath}
which are equipped with the following energy norms,
\begin{displaymath}
  \unorm{u}^2 := \sum_{K \in \MTh} \left( k^2 \| u \|^2_{L^2(K)} + 
  \| \nabla u \|^2_{L^2(K)} \right) + \sum_{e \in \MEh^i \cup \MEh^D} 
  \frac{1}{h_e} \| \jump{u} \|_{L^2(e)}^2, \qquad
  \forall u \in \bmr{V}_h,
\end{displaymath}
and
\begin{displaymath}
  \pnorm{\bm{p}}^2 := \sum_{K \in \MTh} \left( k^2 \| \bm{p} 
  \|^2_{L^2(K)} + \| \nabla \cdot \bm{p} \|^2_{L^2(K)} \right)
  + \sum_{e \in \MEh^i} \frac{1}{h_e} \| \jump{\un \cdot \bm{p}} 
  \|_{L^2(e)}^2, \qquad \forall \bm{p} \in \bmr{\Sigma}_h,
\end{displaymath}
and we define the energy norm $\enorm{\cdot}$ as
\begin{displaymath}
  \enorm{(u, \bm{p})}^2 := \unorm{u}^2 + \pnorm{\bm{p}}^2 + \sum_{e \in
  \MEh^R} \frac{1}{h_e} \| \un \cdot \bm{p} + \ui u \|^2_{L^2(e)}, 
  \qquad \forall (u, \bm{p}) \in \bmr{V}_h \times \bmr{\Sigma}_h.
\end{displaymath}
It is easy to see that $\unorm{\cdot}$, $\pnorm{\cdot}$ and 
$\enorm{\cdot}$ are well-defined norms for their corresponding spaces.

We will derive the error estimates for the numerical solution to 
the problem \eqref{eq_bilinear} under 
the Lax-Milgram framework, which requires us to indicate the
continuity and the coercivity of the bilinear form
\eqref{eq_bilinear}. We first state the continuity result of the
bilinear form $a_h(\cdot; \cdot)$ under the norm $\enorm{\cdot}$.

\begin{lemma}
  Let the bilinear form $a_h(\cdot; \cdot)$ be defined as 
  \eqref{eq_bilinearform}, there exists
  a constant $C$ such that
  \begin{equation}
    | a_h(u, \bm{p}; v, \bm{q}) | \leq C \enorm{(u, \bm{p})} \enorm{(v,
    \bm{q})},
    \label{eq_continuity}
  \end{equation}
  for any $(u, \bm{p}), (v, \bm{q}) \in \bmr{V}_h \times
  \bmr{\Sigma}_h$.
  \label{le_continuity}
\end{lemma}
\begin{proof}
  Using the Cauchy-Schwartz inequality, we have that
  \begin{displaymath}
    \sum_{K \in \MTh} \int_K \nabla \cdot \bm{p}_h \ \overline{\nabla
    \cdot \bm{q}_h} \d{x} \leq \left( \sum_{K \in \MTh} \| \nabla \cdot
    \bm{p}_h \|^2_{L^2(K)} \right)^{\frac{1}{2}} \left( \sum_{K \in
    \MTh} \| \nabla \cdot \bm{q}_h \|^2_{L^2(K)} \right)^{\frac{1}{2}}.
  \end{displaymath}
  Other terms that appear in the bilinear form \eqref{eq_bilinearform}
  can be bounded similarly, which gives us the inequality
  \eqref{eq_continuity} and completes the proof.
\end{proof}

Then we will focus on the coercivity to the bilinear form $a_h(\cdot;
\cdot)$. We first prove the stability property in the continuous 
level by making the use of result \eqref{eq_stability}. In this step, 
the wavenumber $k$ is extracted from the constant that appeared in 
the inequality, which allows us to obtain $k$-explicit error estimates.  

\begin{lemma}
  Let $k_0$ be an arbitrary strictly positive number. For $k \geq
  k_0$, there exists a constant $C$ such that
  \begin{equation}
    \unorm{u} + \pnorm{\bm{p}} \leq C k \left( \|
    \nabla u - k\bm{p} \|_{L^2(\Omega)} + \| \nabla \cdot \bm{p} + ku
    \|_{L^2(\Omega)} + \| \un \cdot \bm{p} + \ui u \|_{L^2(\Gamma_R)}
    \right), 
    \label{eq_HelmholtzInequality}
  \end{equation}
  for all $u \in H_{D}^1(\Omega)$ and $\bm{p} \in H(\div, \Omega)$.
  \label{le_HelmholtzInequality}
\end{lemma}
\begin{proof}
  For any $u \in H_D^1(\Omega)$ and $\bm{p} \in H(\div, \Omega)$, 
  we define
  \begin{displaymath}
    \begin{aligned}
      f_1 := &-\nabla \cdot \bm{p} - ku, \quad 
      \bm{f}_2 := \nabla u - k\bm{p}, \quad \text{ in } \Omega, \\
      &g := \un \cdot \bm{p} + \ui u, \quad \text{ on } \Gamma_R.
    \end{aligned}
  \end{displaymath}
  and let 
  \begin{displaymath}
    a(u,v) := (\nabla u, \nabla v)_{L^2(\Omega)} -
    k^2(u,v)_{L^2(\Omega)} + \ui k(u,v)_{L^2(\Gamma_R)}, \quad \forall
    v \in H^1_D(\Omega).
  \end{displaymath}
  Using the integration by parts, we obtain that
  \begin{equation*}
    a(u,v) = k(f_1, v)_{L^2(\Omega)} + (\bm{f}_2, \nabla
    v)_{L^2(\Omega)} + k(g,v)_{L^2(\Gamma_R)}, \quad \forall
    v \in H^1_D(\Omega).
  \end{equation*}
  We take $v = u + \xi$, where $\xi \in H_D^1(\Omega)$ is the unique
  solution of the adjoint problem:
  \begin{equation}
    a(\xi, \phi) = 2k^2(u, \phi)_{L^2(\Omega)}, \quad \forall \phi \in
    H_D^1(\Omega).
    \label{eq_adjoint}
  \end{equation}
  Then by the stability result \eqref{eq_stability}, we have that
  \begin{equation}
    \| \nabla \xi \|_{L^2(\Omega)} + k \| \xi \|_{L^2(\Omega)} \leq C
    k^2 \| u \|_{L^2(\Omega)}.
    \label{eq_xistability}
  \end{equation}
  From \eqref{eq_adjoint} and \eqref{eq_xistability}, we get that
  \begin{displaymath}
    \begin{aligned}
      \text{Re}(a(u,&u+\xi)) = \| \nabla u \|_{L^2(\Omega)}^2 + k^2 \|
      u \|_{L^2(\Omega)}^2  \vspace{1ex}\\
      &\leq k \| f_1 \|_{L^2(\Omega)} (\| u \|_{L^2(\Omega)} + \| \xi
      \|_{L^2(\Omega)}) + \| \bm{f}_2 \|_{L^2(\Omega)} ( \| \nabla u
      \|_{L^2(\Omega)} + \| \nabla \xi \|_{L^2(\Omega)})
      \vspace{1ex}\\
      &\,\, + k \| g \|_{L^2(\Gamma_R)}(\| u \|_{L^2(\Gamma_R)} + \| \xi
      \|_{L^2(\Gamma_R)}) \vspace{1ex} \\
      & \leq C k (\| f_1 \|_{L^2(\Omega)} + \| \bm{f}_2
      \|_{L^2(\Omega)}) ( \| \nabla u \|_{L^2(\Omega)} + k \| u
      \|_{L^2(\Omega)}) \vspace{1ex}\\ 
      & \,\, + k \| g \|_{L^2(\Gamma_R)} ( \| u \|_{L^2(\Gamma_R)} + 
      \| \xi \|_{L^2(\Gamma_R)}).
    \end{aligned}
  \end{displaymath}
  The rest is to bound the boundary terms $\| u \|_{L^2(\Gamma_R)}$
  and $\| \xi \|_{L^2(\Gamma_R)}$. Taking $\phi = \xi$ in
  \eqref{eq_adjoint} gives us that
  \begin{displaymath}
    \begin{aligned}
      \text{Im}(a(\xi, &\xi)) = k \| \xi \|_{L^2(\Gamma_R)}^2 \leq 2k^2 \|
      \xi \|_{L^2(\Omega)} \| u \|_{L^2(\Omega)} \vspace{1ex}\\
      &\leq k(\| \xi \|_{L^2(\Omega)}^2 + k^2 \| u \|_{L^2(\Omega)}^2),
    \end{aligned}
  \end{displaymath}
  which implies
  \begin{equation}
    \| \xi \|_{L^2(\Gamma_R)} \leq C(\| \xi \|_{L^2(\Omega)} + k \| u
    \|_{L^2(\Omega)}).
    \label{eq_xitrace}
  \end{equation}
  The term $\| u \|_{L^2(\Gamma_R)}$ is bounded by the trace theorem
  \begin{displaymath}
    \begin{aligned}
      \| u \|_{L^2(\Gamma_R)}^2 &\leq C \| u \|_{L^2(\Omega)} \| u
      \|_{H^1(\Omega)} \vspace{1ex} \\
      &\leq C(\frac{k^2}{2} \| u \|_{L^2(\Omega)}^2 + \frac{1}{2k^2} \|
      u \|_{H^1(\Omega)}^2),
    \end{aligned}
  \end{displaymath}
  which implies
  \begin{equation}
    \| u \|_{L^2(\Gamma_R)}^2 \leq C(k\| u \|_{L^2(\Omega)} + \| \nabla
    u \|_{L^2(\Omega)}).
    \label{eq_utrace}
  \end{equation}
  Combing \eqref{eq_xitrace} and \eqref{eq_utrace}, we get
  \begin{displaymath}
    k \| u \|_{L^2(\Omega)} + \| \nabla u \|_{L^2(\Omega)} \leq C k (\|
    f_1 \|_{L^2(\Omega)} + \| \bm{f}_2 \|_{L^2(\Omega)} + \| g
    \|_{L^2(\Gamma_R)}).
  \end{displaymath}
  Further,
  \begin{displaymath}
    \begin{aligned}
      \pnorm{\bm{p}} &\leq C( k \| \bm{p} \|_{L^2(\Omega)} + \| \nabla
      \cdot \bm{p} \|_{L^2(\Omega)}) \vspace{1ex}\\
      &\leq C( \| \nabla u \|_{L^2(\Omega)} + \| \bm{f}_2
      \|_{L^2(\Omega)} + k \| u \|_{L^2(\Omega)} + \| f_1
      \|_{L^2(\Omega)} ) \vspace{1ex}\\
      & \leq Ck (\| f_1 \|_{L^2(\Omega)} + \| \bm{f}_2 \|_{L^2(\Omega)}
      + \| g \|_{L^2(\Gamma_R)}),
    \end{aligned}
  \end{displaymath}
  which gives us the estimate \eqref{eq_HelmholtzInequality} and
  completes the proof.
\end{proof}

We state the following lemmas, together with Lemma
\ref{le_HelmholtzInequality}, to prove the coercivity to the bilinear
form $a_h(\cdot; \cdot)$.
\begin{lemma}
  For any ${u}_h \in \bmr{V}_h^m$, there exists a piecewise polynomial
  function $v_h \in H^1_D(\Omega)$ such that
  \begin{equation}
    \begin{aligned}
      \sum_{K \in \MTh} \left( h_K^{-2} \|u_h - v_h \|_{L^2(K)}^2 + 
      \|\nabla (u_h - v_h) \|_{L^2(K)}^2 \right) \leq C \sum_{e \in 
      \MEh^i \cup \MEh^D} h_e^{-1} \| \jump{u_h}\|_{L^2(e)}^2.
    \end{aligned}
    \label{eq_projection1}
  \end{equation} 
  \label{le_projection1}
\end{lemma}

\begin{proof}
  The proof follows from the techniques as in
  \cite{Karakashian2003post}. For each $K \in \MTh$, let
  $\mc{N}_K = \left\{ \bm{x}_K^{(i)}, \,\, i = 1, \cdots, M \right\}$ be 
  the Lagrange points of $K$ and $\left\{ \varphi_K^{(i)}, \,\, i = 1
  ,\cdots, M \right\}$ be the corresponding Lagrange basis, where $M$
  is the number of degrees of freedom for the Lagrange element of
  order $m$. We set $\mc{N} := \cup_{K \in \MTh} \mc{N}_K$ and
  \begin{equation*}
    \begin{aligned}
      \mc{N}_i &:= \left\{ \nu \in \mc{N} : \exists K \in \MTh, \,\, 
      \nu \text{ is interior to } K \right\}, \\
      \mc{N}_b & := \left\{ \nu \in \mc{N} : \nu \text{ lies on } \Gamma_D
      \right\}, \\
      \mc{N}_{e} & := \mc{N} \backslash (\mc{N}_i \cup \mc{N}_b).
    \end{aligned}
  \end{equation*}
  Let $\omega_{\nu} = \left\{ K \in \MTh | \,\, \nu \in K \right\}$
  and denote its cardinality by $| \omega_{\nu} |$. Since the mesh is
  shape-regular, $|\omega_{\nu}|$ is bounded by a constant.
  For any given $u_h \in \bmr{V}_h^m$, there exists a group of
  coefficients $\{a_K^{(j)}\}$ such that 
  \begin{equation*}
    u_h = \sum_{K \in \MTh} \sum_{1 \leq j \leq M} \alpha_K^{(j)}
    \varphi_K^{(j)}.
  \end{equation*}
  To each node $\nu \in \mc{N}$, we associate the basis function
  $\varphi^{(\nu)}$ given by
  \begin{displaymath}
    \varphi^{(\nu)} |_{K} := \left\{
    \begin{aligned}
      &\varphi_K^{(j)}, && \text{if } \bm{x}_K^{(j)} = \nu, \\
      &0, && \text{otherwise}. \\
    \end{aligned}
    \right.
  \end{displaymath}
  We define $v_h \in \bmr{V}_h^m \cap H^1_D(\Omega)$ by
  \begin{equation*}
    v_h = \sum_{\nu \in \mc{N}} \beta^{(\nu)} \varphi^{(\nu)}.
  \end{equation*}
  where
  \begin{displaymath}
    \beta^{(\nu)} := \left\{\begin{aligned}
      &0, && \text{if } \nu \in \mc{N}_b, \\
      &\frac{1}{|\omega_{\nu}|} \sum_{\bm{x}_K^{(j)} = \nu}
      \alpha_K^{(j)}, && \text{if } \nu \in \mc{N} \backslash \mc{N}_b.
      \\
    \end{aligned} \right.
  \end{displaymath}
  Let $\beta_K^{(j)} = \beta^{(\nu)}$ whenever $\bm{x}_K^{(j)} =
  \nu$. By scaling argument, we have that
  \begin{equation*}
    \| \nabla \varphi_K^{(j)} \|_{L^2(K)}^2 \leq C h_K^{d-2}, \qquad   
    \|\varphi_K^{(j)} \|^2_{L^2(K)} \leq C h_K^d.
  \end{equation*}
  Hence,
  \begin{equation*}
    \begin{aligned}
      \sum_{K \in \MTh} \| \nabla(u_h &- v_h) \|_{L^2(K)}^2 \leq C
      \sum_{K \in \MTh} h_K^{d-2} \sum_{j = 1}^{M} | \alpha_K^{(j)} -
      \beta_K^{(j)} |^2 \\
      &\leq C
      \sum_{\nu \in \mc{N}_{e}} h_{\nu}^{d-2} \sum_{\bm{x}_K^{(j)} = \nu}
      |\alpha_K^{(j)} - \beta^{(\nu)}|^2
      + C \sum_{\nu \in \mc{N}_b} h_{\nu}^{d-2} \sum_{\bm{x}_K^{(j)} = \nu} |
      \alpha_K^{(j)} |^2 \\
      &\leq C \sum_{e \in \MEh^i} h_e^{ d-2 } \sum_{\nu \in e} |
      \alpha_{K^+}^{(j_{\nu}^{+})} - \alpha_{K^-}^{(j_{\nu}^{-})} |^2
      + C \sum_{e \in \MEh^D} h_e^{ d-2 } \sum_{\nu \in e} |
      \alpha_K^{(j_{\nu})} |^2,
    \end{aligned}
  \end{equation*}
  with $h_{\nu} = \max\limits_{K \in \omega_{\nu}} h_K$ and
  $\bm{x}_{K^+}^{(j_{\nu}^{+})} = \bm{x}_{K^-}^{(j_{\nu}^{-})} = \nu$.
  Note that $ |\alpha_{K^+}^{(j_{\nu}^{+})} -
  \alpha_{K^-}^{(j_{\nu}^{-}) }| \leq C\|\jump{u_h} \|_{L^\infty(e)} $,
  together with the inverse inequality, we have
  \begin{equation*}
    \sum_{K \in \MTh} \| \nabla (u_h - v_h) \|_{L^2(K)}^2 \leq C
    \sum_{e \in \MEh^D \cup \MEh^i} h_e^{d-2} \|\jump{u_h}
    \|_{L^{\infty}(e)}^2 \leq C \sum_{e \in \MEh^D \cup \MEh^i}
    h_e^{-1} \| \jump{u_h} \|_{L^2(e)}^2.
  \end{equation*}
  Similarly,
  \begin{equation*}
    \sum_{K \in \MTh} h_K^{-2} \| u_h - v_h \|_{L^2(K)}^2 \leq C
    \sum_{e \in \MEh^i \cup \MEh^D} h_e^{-1} \| \jump{u_h}
    \|_{L^2(e)}^2,
  \end{equation*}
  which deduces to \eqref{eq_projection1} and completes the proof.
\end{proof}

\begin{lemma}
  For any $\bm{p}_h \in \bmr{\Sigma}_h^m$, there exists a piecewise
  polynomial function $\bm{w}_h \in H(\div, \Omega)$ such that 
  \begin{equation}
    \begin{aligned}
      \sum_{K \in \MTh} \left( h_K^{-2} \|\bm{p}_h- \bm{w}_h 
      \|_{L^2(K)}^2 + \|\nabla \cdot ( \bm{p}_h- \bm{w}_h ) 
      \|_{L^2(K)}^2 \right) \leq C \sum_{e \in \MEh^i} h_e^{-1} \|
      \jump{\un \cdot \bm{p}_h }\|_{L^2(e)}^2.
    \end{aligned}
    \label{eq_Hdivprojection}
  \end{equation}
  \label{le_Hdivprojection}
\end{lemma}

\begin{proof}
  We prove the result by using the projection techniques as in
  \cite{Karakashian2003post, Li2020discontinuous}. We will construct a
  new piecewise polynomial function in the \RT space that satisfies 
  the estimate \eqref{eq_Hdivprojection}. We first present some 
  details about the Raviart-Thomas (\RT\!) element, which is the 
  well-known $H(\div,\Omega)$-conforming element proposed in 
  \cite{Raviart1977mixed}. For a bounded domain $D$, we denote by 
  $\wt{\mb{P}}_k(D)$ the set of homogeneous polynomials of degree $k$ 
  on $D$. For the element $K \in \MTh$, the \RT element $\RT_k(K)$ of 
  degree $k$ is given as
  \begin{displaymath}
    \RT_k(K) := \mb{P}_k(K)^d + \bm{x}\wt{\mb{P}}_k(K).
  \end{displaymath}
  For a face $e$, we denote by $\left\{ \bm{q}_e^i \right\}_{i
  =1}^{N_e}$ a basis of the polynomial space $\mb{P}_k(e)$, and for an
  element $K$, we deonte by $\left\{ \bm{q}_K^i \right\}_{i=1}^{N_b}$
  a basis of the polynomial space $\mb{P}_{k - 1}(K)$. For a vector
  field $\bm{v} \in \RT_k(K)$, the moments associated with
  faces of $K$ and $K$ itself are defined as
  \begin{equation}
    \begin{aligned}
      M_K^e(\bm{v}) & := \left\{ \int_e (\un_e \cdot \bm{v})
      \bm{q}_e^i \d{s} \right\}, \quad \text{for any face } e \in
      \mc{E}(K), \\
      M_K^b(\bm{v}) & := \left\{ \int_K \bm{v} \cdot \bm{q}_K^i \d{x}
      \right\},\\
    \end{aligned}
    \label{eq_RTmoments}
  \end{equation}
  where $\mc{E}(K)$ denotes the set of faces of the element $K$.
  The polynomials in $\RT_k(K)$ can be uniquely determined by the
  moments given in \eqref{eq_RTmoments}\cite{Raviart1977mixed}. For
  any $\bm{q} \in \RT_k(K)$, we define $q_{K, e}^i \in M_K^e(\bm{q})(1
  \leq i \leq N_e)$, $q_{K,b}^i \in M_K^b(\bm{q})(1 \leq i \leq N_b)$
  as its corresponding moments, respectively.  We denote by $\left\{
  \bm{\phi}_{K, e}^i \right\}(1 \leq i \leq N_e)$ and $\left\{
  \bm{\phi}_{K,b}^i \right\}(1 \leq i \leq N_b)$ the basis functions
  of $\RT_k(K)$ with respect to the moments $M_K^e(\cdot)$ and
  $M_K^b(\cdot)$, respectively. From $\left\{ \bm{\phi}_{K, e}^i
  \right\}$ and $\left\{ \bm{\phi}_{K,b}^i \right\}$ and the moments
  in \eqref{eq_RTmoments}, any polynomial $\bm{q} \in \RT_k(K)$ can be
  expressed as
  \begin{displaymath}
    \bm{q} = \sum_{e \in \mc{E}(K)} \sum_{i = 1}^{N_e} q_{K, e}^i
    \bm{\phi}_{K, e}^i + \sum_{i = 1}^{N_b} q_{K, b}^i \bm{\phi}_{K,
    b}^i.
  \end{displaymath}
  Then we will take advantage of the affine equivalence of elements.
  Let $\wh{K}$ be the reference simplex of $d$ dimensions and we
  employ the Piola transformation that maps a vector $\wh{\bm{v}}:
  \wh{K} \rightarrow \mb{R}^d$ to a vector $\bm{v}: K \rightarrow
  \mb{R}^d$. The Piola transformation preserves the moments and we
  refer to \cite{Raviart1977mixed, Brezzi1991mixed} for detailed
  properties of the Piola transformation. Then we have that
  \begin{equation}
    h_K^{-2}\|{\bm{q}} \|_{L^2({K})}^2 + \|\nabla \cdot {\bm{q}}
    \|_{L^2({K})}^2 \leq C   h_K^{-d}  \left( \sum_{e \in \mc{E}(K)}
    \sum_{i = 1}^{N_e} (q_{K, e}^i)^2 + \sum_{i = 1}^{N_b} (q_{K,
    b}^i)^2 \right).
    \label{eq_Kdivmoments}
  \end{equation}
  It is clear that \eqref{eq_Kdivmoments} holds on the reference
  element. On a general element $K$, we obtain the estimate
  \eqref{eq_Kdivmoments} from the properties of the Piola
  transformation, $\|\bm{q}\|_{L^2(K)}^2 \leq C h_K^{-d + 2}
  \|\wh{\bm{q}} \|_{L^2(\wt{K})}^2$ and  $\| \nabla \cdot
  \bm{q}\|_{L^2(K)}^2 \leq Ch_K^{-d} \|\wh{\nabla} \cdot \wh{\bm{q}}
  \|_{L^2(\wt{K})}^2$. We let $e \in \MEh^i$ be an interior face
  shared by two adjacent elements $K_1$ and $K_2$. For two polynomials
  $\bm{q}_1 \in \RT_k(K_1)$ and $\bm{q}_2 \in \RT_k(K_2)$, we state
  that there exists a constant $C$ such that 
  \begin{equation}
    \sum_{i = 1}^{N_e} (q_{K_1, e}^i - q_{K_2, e}^i)^2 \leq C h_e^{d -
    1} \int_e (\un \cdot(\bm{q}_1 - \bm{q}_2))^2 \d{s}.
    \label{eq_q1q2f}
  \end{equation}
  We also apply the scaling argument to obtain \eqref{eq_q1q2f}. We
  first assume that both $K_1$ and $K_2$ are of the reference size.
  We note that the left hand side of \eqref{eq_q1q2f} vanishes implies
  that the right hand side of \eqref{eq_q1q2f} also equals to zero and
  vice verse. The estimate \eqref{eq_q1q2f} holds due to the
  equivalence of norms over finite dimensional spaces. For general
  cases, we can obtain \eqref{eq_q1q2f} from the scaling estimate $\|
  \bm{\wh{q}} \|_{L^2(e)}^2 \leq C h_K^{d - 1} \| \bm{q}
  \|_{L^2(e)}^2$. 

  Now we are ready to prove Lemma \ref{le_Hdivprojection} by
  constructing a new piecewise polynomial $\bm{w}_h \in H(\div,
  \Omega)$ satisfying \eqref{eq_Hdivprojection}. Clearly,
  $\mb{P}_k(K)^d \subset \RT_k(K)$ for any $K \in \MTh$ and we let $\{
  p_{K, e}^i \}$ and $\{ p_{K, b}^i\}$ be the moments of $\bm{p}_h$
  for any $K \in \MTh$ and any $e \in \mc{E}(K)$. We construct
  $\bm{w}_h$ by defining the following moments on faces and elements:
  \begin{equation}
    w_{K, e}^i := \frac{1}{|N(e)|} \sum_{K' \in N(e)} p_{K', e}^i,
    \quad 1 \leq i \leq N_e, \quad \forall e \in \MEh, 
    \label{eq_wKfi}
  \end{equation}
  and 
  \begin{equation}
    w_{K, b}^i = p_{K, b}^i, \quad 1 \leq i \leq N_b, \quad \forall K
    \in \MTh,
    \label{eq_wKbi}
  \end{equation}
  where $N(e) := \left\{ K' \in \MTh \ | \  e \subset \mc{E}(K')
  \right\}$ and $|N(e)|$ denotes the cardinality of $N(e)$. Obviously,
  $1 \leq |N(e)| \leq 2$ and $|N(e)| = 1$ implies $e \in \MEh^b$.  By
  the property of $\RT_k$ space, $\bm{w}_h \in H(\div, \Omega)$ from
  these moments. The rest is to bound $\bm{p}_h - \bm{w}_h$.  On the
  element $K$, by \eqref{eq_Kdivmoments} and \eqref{eq_wKfi} we have
  that 
  \begin{displaymath}
    h_K^{-2} \|\bm{p}_h - \bm{w}_h \|_{L^2(K)}^2 + \| \nabla \cdot
    (\bm{p}_h - \bm{w}_h) \|_{L^2(K)}^2 \leq 
    C h_K^{-d} \left( \sum_{e \in \mc{E}(K)} 
    \sum_{i = 1}^{N_e}(p_{K, e}^i - w_{K, e}^i)^2\right).
  \end{displaymath}
  On the boundary face $e$, $\bm{p}_h$ and $\bm{w}_h$ clearly have the
  same moments on $e$.  A summation over all elements, together with
  the mesh regularity and \eqref{eq_wKfi} and \eqref{eq_q1q2f}, gives
  that 
  \begin{displaymath}
    \begin{aligned}
      \sum_{K \in \MTh} \big( h_K^{-2}  \|\bm{p}_h &- \bm{w}_h
      \|_{L^2(K)}^2 + \|\nabla \cdot ( \bm{p}_h- \bm{w}_h )\|_{L^2(K)}^2 
       \big) \leq C \sum_{e \in \MEh}\sum_{i=1}^{N_e}  
      h_e^{-d} (p_{K, e}^i - w_{K, e}^i)^2  \\
      & \leq  C \sum_{e \in \MEh^i} \sum_{i = 1}^{N_e} h_e^{-d}
      \left(p_{K, e}^i - \frac{p_{K, e}^i + p_{K', e}^i}{2} \right)^2
      \quad (e \text{ is shared by } K \text{ and } K') \\
      &\leq C \sum_{e \in \MEh^i} \sum_{i = 1}^{N_e} h_e^{-d}
      \left(p_{K, e}^i - p_{K', e}^i \right)^2 \leq C \sum_{e \in
      \MEh^i} h_e^{-1} \| \jump{\un \cdot \bm{p}_h} \|_{L^2(e)}^2.
    \end{aligned}
  \end{displaymath}
  This gives the estimate \eqref{eq_Hdivprojection} and completes the
  proof.
\end{proof}

Now we are ready to state that the bilinear form $a_h(\cdot; \cdot)$ 
is coercive under the energy norm $\enorm{\cdot}$. 
\begin{lemma}
  Let the bilinear form $a_h(\cdot, \cdot)$ be defined as 
  \eqref{eq_bilinearform}, there exists
  a constant $C$ such that
  \begin{equation}
    a_h(u_h, \bm{p}_h; u_h, \bm{p}_h) \geq C k^{-2} (1 + h + k^2 h^2)^{-1}
    \enorm{(u_h, \bm{p}_h)}^2,
    \label{eq_coercivity}
  \end{equation}
  for any $(u_h, \bm{p}_h) \in \bmr{V}_h^m \times \bmr{\Sigma}_h^m$.
  \label{le_coercivity}
\end{lemma}

\begin{proof}
  Clearly, we have that
  \begin{displaymath}
    \begin{aligned}
      a_h(u_h, \bm{p}_h; u_h&, \bm{p}_h) = \sum_{K \in \MTh} \left(
      \| \nabla u_h - k \bm{p}_h \|^2_{L^2(K)} + \| \nabla \cdot
      \bm{p}_h + k u_h \|^2_{L^2(K)} \right) \\
      &+ \sum_{e \in \MEh^i} \frac{1}{h_e} \left( \| \jump{u_h}
      \|^2_{L^2(e)} + \| \jump{\bm{p}_h} \|^2_{L^2(e)} \right) \\
      &+ \sum_{e \in \MEh^D} \frac{1}{h_e} \| u_h \|^2_{L^2(e)}
      + \sum_{e \in \MEh^R} \frac{1}{h_e} \| \un \cdot \bm{p}_h +
      \ui u_h \|^2_{L^2(e)}.
    \end{aligned}
  \end{displaymath}
  By Lemma \ref{le_projection1}, there exists a polynomial
  $v_h \in \bmr{V}_h^m \cap H^1_D(\Omega)$ and a polynomial
  $\bm{q}_h \in \bmr{\Sigma}_h^m \cap H(\div,\Omega)$, such that
  \begin{displaymath}
    \unorm{u_h - v_h}^2 \leq C \sum_{e \in \MEh^i \cup \MEh^D}
    (h_e^{-1} + k^2 h_e) \| \jump{u_h} \|^2_{L^2(e)} \leq
    C (1 + k^2 h^2) a_h(u_h,\bm{p}_h; u_h, \bm{p}_h),
  \end{displaymath}
  and
  \begin{displaymath}
    \pnorm{\bm{p}_h - \bm{q}_h}^2 \leq C \sum_{e \in \MEh^i}
    (h_e^{-1} + k^2 h_e) \| \jump{\bm{p}_h} \|^2_{L^2(e)}
    \leq C (1 + k^2 h^2) a_h(u_h, \bm{p}_h; u_h, \bm{p}_h).
  \end{displaymath}
  Hence,
  \begin{displaymath}
    \begin{aligned}
      \enorm{(u_h, \bm{p}_h)}^2 &\leq C \left(\unorm{u_h - v_h}^2 +
      \pnorm{\bm{p}_h - \bm{q}_h}^2 + \unorm{v_h}^2 +
      \pnorm{\bm{q}_h}^2 + \sum_{e \in \MEh^R} \frac{1}{h_e}
      \| \un \cdot \bm{p}_h + \ui u_h \|^2_{L^2(e)}
      \right) \\
      &\leq C \left( (1+k^2h^2) a_h(u_h, \bm{p}_h; u_h, \bm{p}_h)
      + \unorm{v_h}^2 + \pnorm{\bm{q}_h}^2
      + \sum_{e \in \MEh^R} \frac{1}{h_e} \| \un \cdot \bm{p}_h + \ui
      u_h \|^2_{L^2(e)} \right).
    \end{aligned}
  \end{displaymath}
  By Lemma \ref{le_HelmholtzInequality}, we get that
  \begin{displaymath}
    (\unorm{v_h} + \pnorm{\bm{q}_h})^2 \leq C k^2 \left(\| \nabla v_h -
    k\bm{q}_h \|_{L^2(\Omega)} + \| \nabla \cdot \bm{q}_h + kv_h
    \|_{L^2(\Omega)} + \| \un \cdot \bm{q}_h + \ui v_h
    \|_{L^2(\Gamma_R)} \right)^2.
  \end{displaymath}
  We apply the triangle inequality to derive that
  \begin{displaymath}
    \begin{aligned}
      \| \nabla v_h - k\bm{q}_h \|^2_{L^2(\Omega)} &\leq C \left( \|
      \nabla u_h - k\bm{p}_h \|^2_{L^2(\MTh)} + \| \nabla (u_h - v_h)
      \|^2_{L^2(\MTh)} + k^2 \|\bm{p}_h - \bm{q}_h \|^2_{L^2(\MTh)}
      \right) \\
      &\leq C\left(\| \nabla u_h - k\bm{p}_h \|^2_{L^2(\MTh)} +
      \unorm{u_h - v_h}^2 + \pnorm{\bm{p}_h - \bm{q}_h}^2 \right) \\
      &\leq C (1 + k^2 h^2) a_h(u_h, \bm{p}_h; u_h, \bm{p}_h).
    \end{aligned}
  \end{displaymath}
  Similarly,
  \begin{displaymath}
    \| \nabla \cdot \bm{q}_h + kv_h \|^2_{L^2(\Omega)} \leq C (1 +
    k^2 h^2) a_h(u_h, \bm{p}_h; u_h, \bm{p}_h).
  \end{displaymath}
  From the proof of Lemma \ref{le_Hdivprojection}, $\bm{p}_h$ and
  $\bm{q}_h$ has the same moments on any boundary face $e$, which
  implies $\|\un \cdot \bm{q}_h - \un \cdot \bm{p}_h \|_{L^2(e)} = 0$
  for any $e \in \MEh^R$. Together with the triangle inequality, 
  we have that 
  \begin{displaymath}
    \| \un \cdot \bm{q}_h + \ui v_h \|_{L^2(\Gamma_R)}^2 \leq \sum_{e
    \in \MEh^R} \left( \| \un \cdot \bm{p}_h + \ui u_h \|_{L^2(e)}^2 +
    \| u_h - v_h \|_{L^2(e)}^2
    \right).
  \end{displaymath}
  The trace inequality gives us 
  \begin{displaymath}
    h_e^{-1} \| u_h-v_h \|^2_{L^2(e)} \leq C \left( h_e^{-2}
    \|u_h-v_h\|^2_{L^2(K)} + \| \nabla(u_h-v_h) \|^2_{L^2(K)} \right),
    \qquad \forall e \in \MEh^R,
  \end{displaymath}
  where $K$ is an element such that $e \in \mc{E}(K)$. We apply Lemma
  \ref{le_projection1} to conclude that 
  \begin{displaymath}
    \|\un \cdot \bm{q}_h + \ui v_h \|_{L^2(\Gamma_R)}^2 \leq C h
    a_h(u_h, \bm{p}_h; u_h, \bm{p}_h).
  \end{displaymath}
  Combining all the inequalities above, we arrive at
  \begin{displaymath}
    a_h(u_h, \bm{p}_h; u_h, \bm{p}_h) \geq C k^{-2} (1 + h + k^2 h^2)^{-1}
    \enorm{(u_h, \bm{p}_h)}^2,
  \end{displaymath}
  which gives the estimate \eqref{eq_coercivity} and completes the proof.
\end{proof}

In addition, the bilinear form $a_h(\cdot; \cdot)$ satisfies the
Galerkin orthogonality:
\begin{lemma}
  Let the bilinear form $a_h(\cdot; \cdot)$ be defined as
  \eqref{eq_bilinearform}. Let $(u, \bm{p})
  \in H^1(\Omega) \times H(\div, \Omega)$ be the exact solution to
  \eqref{eq_firstHelmholtz}, and let $(u_h, \bm{p}_h) \in \bmr{V}_h^m
  \times \bmr{\Sigma}_h^m$ be the solution to \eqref{eq_bilinear}. 
  Then, the following equation holds true
  \begin{equation}
    a_h(u - u_h, \bm{p} - \bm{p}_h; v_h, \bm{q}_h) = 0,
    \label{eq_orthogonality}
  \end{equation}
  for any $(v_h, \bm{q}_h) \in \bmr{V}_h^m \times \bmr{\Sigma}_h^m$.
  \label{le_othogonality}
\end{lemma}

\begin{proof}                                                         
  The regularity of the exact solution $(u, \bm{p})$ directly brings  
  us that                                                             
  \begin{equation*}                                                   
    \jump{u} = 0, \qquad \jump{\un \cdot \bm{p}} = 0, \qquad \text{on }
    \forall e \in \MEh^i.                                             
  \end{equation*}                                                     
  Hence,                                                              
  \begin{align*}                                                   
    a_h(u-u_h, &\bm{p}-\bm{p}_h; v_h, \bm{q}_h) = \sum_{K \in \MTh} 
    \int_K (\nabla \cdot (\bm{p} - \bm{p}_h) + k(u - u_h)) \ 
    \overline{(\nabla \cdot \bm{q}_h + k v_h)} \d{x} \\
    &+ \sum_{K \in \MTh} \int_K (\nabla (u - u_h) - k(\bm{p} -
    \bm{p}_h)) \cdot \overline{(\nabla v_h - k \bm{q}_h)} \d{x} \\
    &- \sum_{e \in \MEh^i} \frac{1}{h_e} \int_e \jump{u_h}\ 
    \overline{\jump{v_h}} \d{s} - \sum_{e \in \MEh^i}
    \frac{1}{h_e} \int_e \jump{\un \cdot \bm{p}_h} \ 
    \overline{\jump{\un \cdot \bm{q}_h}} \d{s} \\
    &+ \sum_{e \in \MEh^D} \frac{1}{h_e} \int_e (u-u_h)\ 
    \overline{v_h} \d{s} + \sum_{e \in \MEh^R} \frac{1}{h_e} \int_e
    (\un \cdot (\bm{p} - \bm{p}_h) + \ui(u-u_h)) \ \overline{(\un 
    \cdot \bm{q}_h + \ui v_h)} \d{s} \\
    & = -\sum_{K \in \MTh} \int_K \wt{f} \ 
    \overline{(\nabla \cdot \bm{q}_h + k v_h)} \d{x}
    + \sum_{e \in \MEh^D} \frac{1}{h_e}
    \int_e g_0 \ \overline{v_h} \d{s} \\
    &\quad + \sum_{e \in \MEh^R} \frac{1}{h_e} \int_e \wt{g}      
    \ \overline{\un \cdot \bm{q}_h + \ui v_h} \d{s}
    - a_h(u_h, \bm{p}_h;v_h, \bm{q}_h) \\
    & = l_h( v_h, \bm{q}_h) - a_h(u_h, \bm{p}_h; v_h, \bm{q}_h) \\  
    & = 0,                                                          
  \end{align*}                                                     
  which yields the equation \eqref{eq_orthogonality}                  
  and completes the proof.
\end{proof}

Finally, we arrive at the {\it a priori} error estimate (with respect
to a fixed wavenumber $k$) of the method under the energy norm
$\enorm{\cdot}$.
\begin{theorem}
  Let $(u, \bm{p}) \in H^{m+1}(\Omega) \times H^{m+1}(\Omega)^d$ 
  be the exact solution to \eqref{eq_firstHelmholtz}. Let $(u_h,
  \bm{p}_h) \in \bmr{V}_h^m \times \bmr{\Sigma}_h^m$ be the numerical
  solution to \eqref{eq_bilinear}. Then there exists a constant $C$
  such that 
  \begin{equation}
    \enorm{(u-u_h, \bm{p}-\bm{p}_h)} \leq C k^2 (1+h+k^2h^2) (1+ k^2
    h^2)^{\frac{1}{2}} h^m (\| u \|_{H^{m+1}(\Omega)} + \| \bm{p} \|
    _{H^{m+1}(\Omega)}).
    \label{eq_estimate}
  \end{equation}
  \label{th_estimate}
\end{theorem}
\begin{proof}
  By Lemma \ref{le_othogonality}, we have that
  \begin{displaymath}
    a_h(u - u_h, \bm{p} - \bm{p}_h; v_h, \bm{q}_h) = 0, \quad \forall 
    (v_h, \bm{q}_h) \in \bmr{V}_h^m \times \bmr{\Sigma}_h^m
  \end{displaymath}
  Together with Lemma \ref{le_coercivity} and Lemma
  \ref{le_continuity}, we obtain that
  \begin{displaymath}
    \begin{aligned}
      \enorm{(u_h-v_h, \bm{p}_h-\bm{q}_h)}^2 &\leq C k^2 (1 + h + 
      k^2 h^2) a_h(u_h-v_h, \bm{p}_h - \bm{q}_h;
      u_h-v_h, \bm{p}_h - \bm{q}_h) \\
      &= C k^2 (1 + h + k^2 h^2)
      a_h(u-v_h, \bm{p}-\bm{q}_h; u_h-v_h, \bm{p}_h - \bm{q}_h)
      \\
      &\leq C k^2(1 + h + k^2 h^2) \enorm{(u-v_h, \bm{p}-\bm{q}_h)}
      \enorm{(u_h-v_h, \bm{p}_h - \bm{q}_h)},
    \end{aligned}
  \end{displaymath}
  for any $(v_h, \bm{q}_h) \in \bmr{V}_h^m \times \bmr{\Sigma}_h^m$.
  We eliminate the term $\enorm{(u_h-v_h, \bm{p}_h-\bm{q}_h)}$
  on both sides and apply the triangle inequality to get that
  \begin{equation}
    \enorm{(u-u_h, \bm{p}-\bm{p}_h)} \leq C k^2(1 + h + k^2 h^2) \inf_{(v_h,
    \bm{q}_h) \in \bmr{V}_h^m \times \bmr{\Sigma}_h^m} \enorm{(u -
    v_h, \bm{p} - \bm{q}_h)}.
    \label{eq_Cea}
  \end{equation}
  We denote by $u_I \in \bmr{V}_h^m$ the standard Lagrange interpolant
  of the exact solution $u$, and by $\bm{p}_I \in \bmr{\Sigma}_h$ the
  BDM interpolant of the exact solution $\bm{p}$. We refer to
  \cite{ciarlet2002finite} and \cite{Brezzi1985two} for details of two
  interpolation operators. By the approximation properties of these
  interpolant operators, we get that
  \begin{equation}
    \begin{aligned}
      &\| u-u_I \|_{L^2(\Omega)} \leq Ch^{m+1} \|u\|_{H^{m+1}(\Omega)},
      \quad \| \nabla(u-u_I) \|_{L^2(\Omega)} \leq Ch^m \|u\|_{H^{m+1}
      (\Omega)}, \\
      &\| \bm{p}-\bm{p}_I \|_{L^2(\Omega)} \leq C h^{m+1}
      \|\bm{p}\|_{H^{m+1}(\Omega)}, \quad \| \nabla \cdot (\bm{p}-
      \bm{p}_I) \|_{L^2(\Omega)} \leq Ch^{m} \|\nabla \cdot \bm{p}\|
      _{H^{m}(\Omega)}.
    \end{aligned}
    \label{eq_interpolant}
  \end{equation}
  We refer to \cite[Theorem 3.2.1]{ciarlet2002finite} and
  \cite[Proposition 2.5.4]{boffi2013mixed} for the proof of
  these inequalities. Since $u_I \in H^1(\Omega)$ and $\bm{p}_I \in
  H(\div, \Omega)$, we have
  \begin{equation}
    \jump{u-u_I} = 0, \quad
    \jump{\un \cdot (\bm{p} - \bm{p}_I)} = 0, \qquad \text{ on }
    \forall e \in \MEh^i.
    \label{eq_jump}
  \end{equation}
  The trace inequality brings us that
  \begin{equation}
    h_e^{-1} \| u-u_I \|^2_{L^2(e)} \leq C \left( h_e^{-2}
    \|u-u_I\|^2_{L^2(K)} + \| \nabla(u-u_I) \|^2_{L^2(K)} \right),
    \qquad \forall e \in \MEh^b,
    \label{eq_trace}
  \end{equation}
  where $K$ is an element having $e$ as a face.
  Denote by $\Pi_h^0$ the $L^2$ projection onto $\bmr{\Sigma}^m_h$. 
  Using \eqref{eq_trace} and the inverse inequality , we derive that
  \begin{equation}
    \begin{aligned}
      h_e^{-1} \| \un \cdot (\bm{p} - \bm{p}_I) &+ \ui (u-u_I)
      \|^2_{L^2(e)} \leq C h_e^{-1} \left( \| \bm{p} - \bm{p}_I
      \|^2_{L^2(e)} + \| u-u_I \|^2_{L^2(e)} \right) \\
      & \leq C h_e^{-1} \left( \| \Pi_h^0(\bm{p} -
      \bm{p}_I)\|^2_{L^2(e)} + \| \bm{p} - \Pi_h^0 \bm{p} \|^2_{L^2(e)}
      + \|u-u_I\|^2_{L^2(e)} \right) \\
      &\leq C \left( h_e^{-2} \| \bm{p}-\bm{p}_I \|^2_{L^2(K)} +
      h_e^{-2} \| u-u_I \|^2_{L^2(K)} + \| \nabla(u-u_I) \|^2_{L^2(K)}
      \right. \\
      &\ \  \left. + h_e^{-1} \|\bm{p} - \Pi_h^0 \bm{p}
      \|^2_{L^2(e)} \right).
    \end{aligned}
    \label{eq_Robin}
  \end{equation}
  Combining with \eqref{eq_interpolant}, \eqref{eq_jump},
  \eqref{eq_trace}, \eqref{eq_Robin} and the approximation property of
  the $L^2$ projection \cite[lemma 4.3]{Houston2005interior}, we arrive
  at
  \begin{displaymath}
    \enorm{(u-u_I, \bm{p}-\bm{p}_I)}^2 \leq C(1+k^2h^2)h^{2m}(\| u
    \|_{H^{m+1}(\Omega)}^2 + \| \bm{p} \|_{H^{m+1}(\Omega)}^2)
  \end{displaymath}
  Let $v_h = u_I$ and $\bm{q}_h = \bm{p}_I$ in \eqref{eq_Cea}, then 
  the above estimate gives the error estimate \eqref{eq_estimate}, 
  which completes the proof.
\end{proof}

\begin{remark}
  We have proved that the numerical solution $(u_h, \bm{p}_h)$ of our
  method has the optimal convergence rate under the energy norm
  $\enorm{\cdot}$. By the definition of the energy norm, the error 
  under the $L^2$ norm for both variables has at least sub-optimal 
  convergence rate, i.e.
  \begin{displaymath}
    \begin{aligned}
      \|u-u_h\|_{L^2(\Omega)} &+ \| \bm{p}-\bm{p}_h \|_{L^2(\Omega)}
      \\
      &\leq C k (1+h+k^2h^2) (1+ k^2h^2)^{\frac{1}{2}} h^m 
      (\| u \|_{H^{m+1}(\Omega)} + \| \bm{p} \|_{H^{m+1}(\Omega)}).
    \end{aligned}
  \end{displaymath}
  It can be seen that the degree of $k$ in the $L^2$ error estimate is
  one less than that in the error estimate under the energy norm
  $\enorm{\cdot}$. 
  In numerical experiments in the next section, we observe the optimal
  convergence rate for the variable $u$ and sub-optimal convergence
  rate for the variable $\bm{p}$ for the $L^2$ error. 
\end{remark}

Another advantage of our method is that the least squares functional
\eqref{eq_functional} can provide a natural mesh refinement
indicator $\eta_K$ for any element $K$, which is defined by
\begin{equation}
  \begin{aligned}
    \eta_K^2 := &\| \nabla \cdot \bm{p}_h + ku_h + \wt{f} \|_{L^2(K)}^2
    + \| \nabla u_h - k \bm{p}_h \|_{L^2(K)}^2 \\
    &+ \sum_{e \in \MEh^i \cap \mc{E}(K)} \frac{1}{h_e} ( \|
    \jump{u_h} \|_{L^2(e)}^2 + \| \jump{\un \cdot \bm{p}_h}
    \|_{L^2(e)}^2) \\
    &+ \sum_{e \in \MEh^D \cap \mc{E}(K)} \frac{1}{h_e} \| u_h - g_0
    \|_{L^2(e)}^2  + \sum_{e \in \MEh^R \cap \mc{E}(K)} \frac{1}{h_e}
    \| \un \cdot \bm{p}_h + \ui u_h - \wt{g} \|_{L^2(e)}^2.
  \end{aligned}
  \label{eq_etaK}
\end{equation}
where $\mc{E}(K)$ is the $d-1$ dimensional faces of $K$. We have the 
following lemma to show that the indicator is exact with respect to
the energy norm $\enorm{\cdot}$.
\begin{lemma}
  Let $(u, \bm{p})$ be the exact solution to
  \eqref{eq_firstHelmholtz}, and let $(u_h, \bm{p}_h) \in
  \bmr{V}_h^m \times \bmr{\Sigma}_h^m$ be the numerical solution to
  \eqref{eq_bilinear}. Then there exists a constant $C$ such that
  \begin{equation}
    \sum_{K \in \MTh} \eta_K^2 \leq C \enorm{(u - u_h,
    \bm{p} - \bm{p}_h)} ^2.
    \label{eq_estimatorup}
  \end{equation}
  \label{le_estimatorup}
\end{lemma}
\begin{proof}
  From the definition of $\eta_K$, it is easy to see that $\sum_{K \in
  \MTh} \eta_K^2 \leq C a_h(u - u_h, \bm{p} - \bm{p}_h; u - u_h, 
  \bm{p} - \bm{p}_h)$.  The estimate \eqref{eq_estimatorup} directly 
  follows from the boundedness property \eqref{eq_continuity}.
\end{proof}

The adaptive procedure consists of loops of the standard form:
\begin{displaymath}
  \text{Solve} \  \rightarrow \  \text{Estimate}\  \rightarrow \
  \text{Mark} \ \rightarrow \  \text{Refine}.
\end{displaymath}
The longest-edge bisection algorithm is used to adaptively refine the
mesh and the detailed adaptive procedure is presented as follow:
\begin{enumerate}[Step 1]
  \item Given the initial mesh $\mc{T}_0$ and a positive parameter
    $\lambda$, and set the iteration number $l = 0$;
  \item Solve the Helmholtz equations on the mesh
    $\mc{T}_l$;
  \item Obtain the error indicator $\eta_K$ for all $K \in \mc{T}_l$
    with respect to the numerical solutions from the Step 2;
  \item Find the minimal subset $\mc{M} \subset
    \mc{T}_l$ such that $\lambda \sum_{K \in \mc{T}_l} \eta_K^2 \leq
    \sum_{K \in \mc{M}} \eta_K^2$ and mark all elements in $\mc{M}$.
  \item Refine all marked elements to generate the next level mesh
    $\mc{T}_{l + 1}$;
  \item If the stop criterion is not satisfied, then go to the Step 2
    and set $l = l + 1$.
\end{enumerate}

%% file: numericalresults.tex
\section{Numerical Results}
\label{sec_numericalresults} 
In this section, we present several numerical examples in two and
three dimensions to demonstrate the performance of the proposed method.
We assume that the domain $D = \emptyset$ without indication, so the 
Dirichlet boundary is empty. 
We adopt the BiCGstab solver together with the ILU preconditioner to 
solve the resulting linear algebraic system.
\begin{figure}
  \centering
  \includegraphics[width=0.4\textwidth]{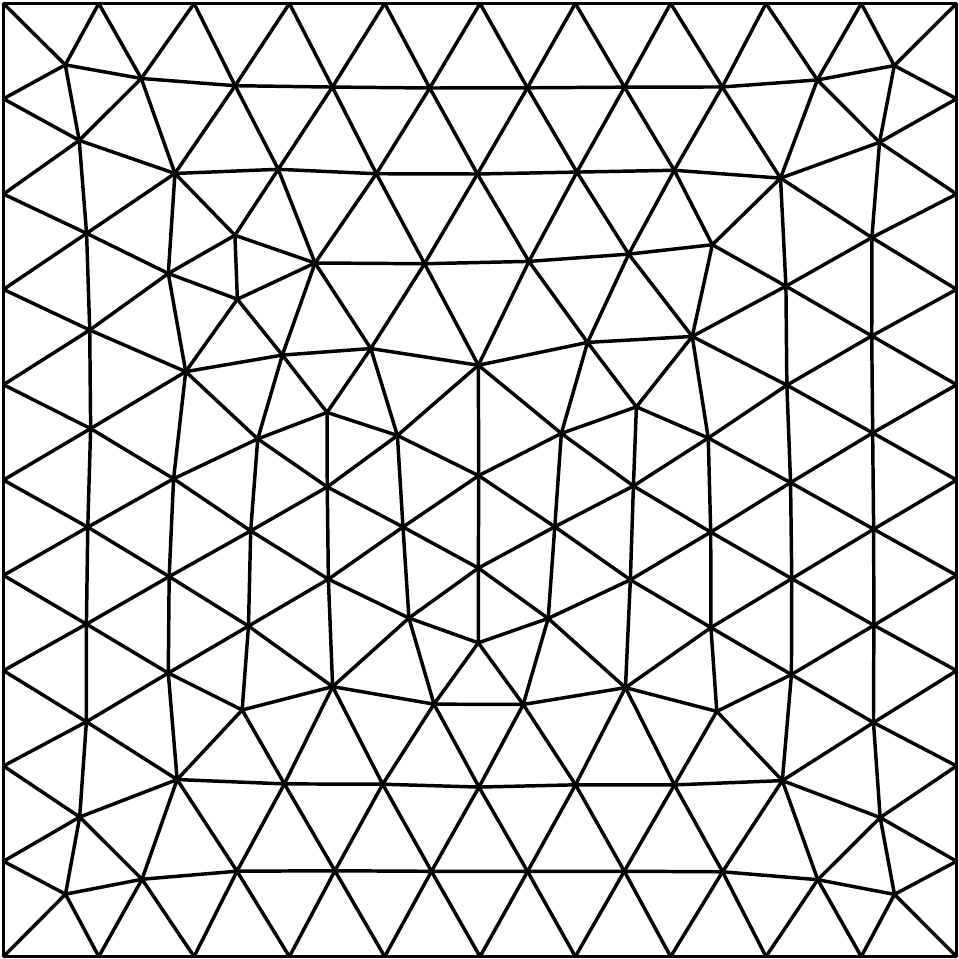}
  \hspace{25pt}
  \includegraphics[width=0.4\textwidth]{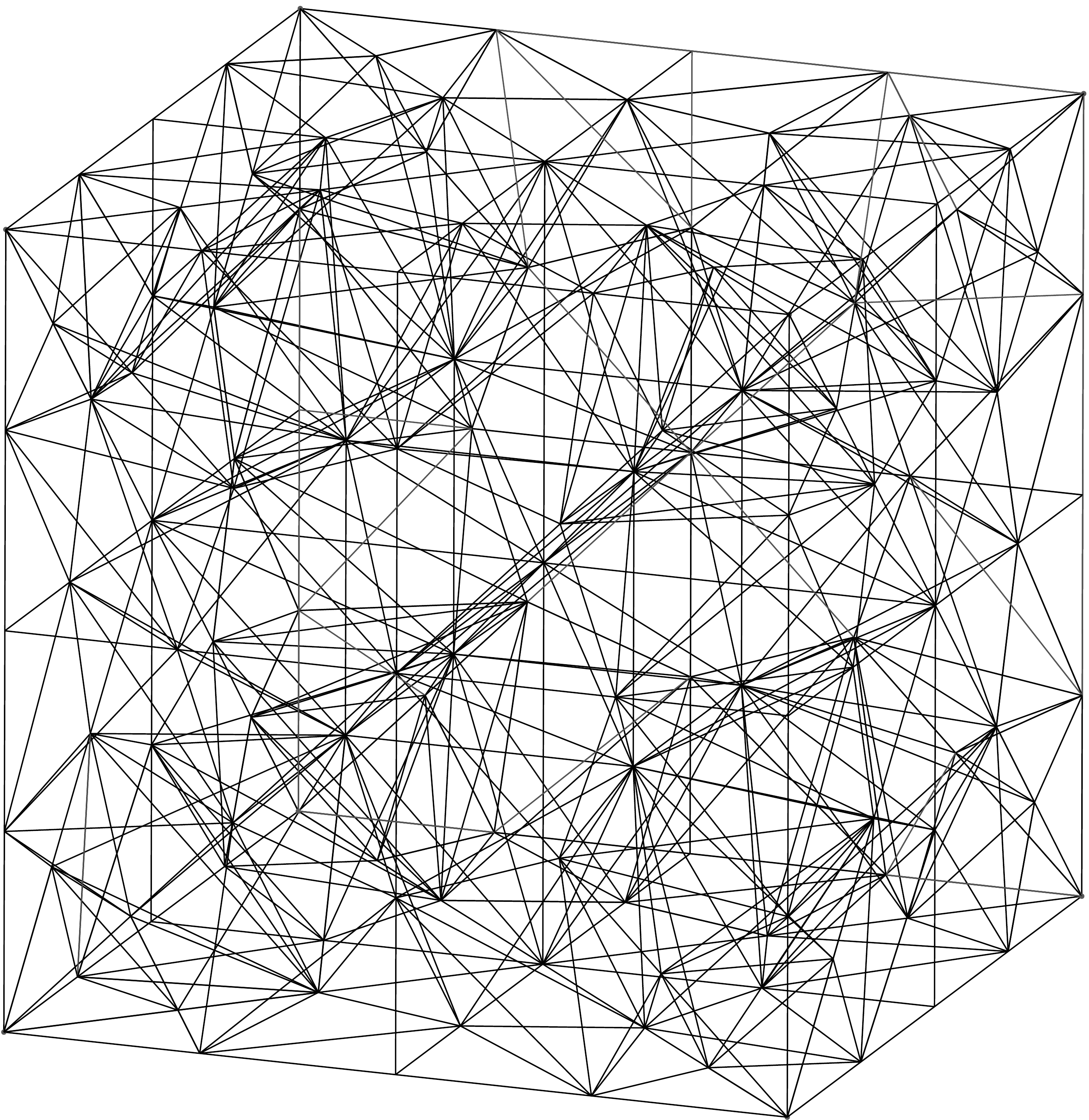}
  \caption{2d triangular partition with $ h = 1/10$ (left) / 3d
  tetrahedral partition with $h = 1/4$ (right).}
  \label{fig_partition}                     
\end{figure}

\noindent \textbf{Example 1.} First, we consider a smooth problem
defined on the unit square domain $\Omega = (0,1)^2$. The exact
solution for the Helmholtz equation is given by \cite{Lee2000first},
\begin{displaymath}
  u(x,y) = \mr{e}^{\ui k (x \cos{\frac{\pi}{5}} + y
  \sin{\frac{\pi}{5}})},
\end{displaymath}
where the source term $f$ and the Robin boundary data $g$ are chosen
accordingly. To obtain the convergence order, we solve this problem on
a series of shape-regular meshes with the mesh size 
$h = 1/5$, $h = 1/10$, $\ldots$, $1/40$, see Fig~.\ref{fig_partition}.  
The convergence histories with the wave number $k = 1,2,8$ for the 
accuracy $m = 1, 2, 3, 4$ are present in Tab.~\ref{tab_ex1k1}, 
Tab.~\ref{tab_ex1k2} and Tab.~\ref{tab_ex1k8}, respectively. 
From the numerical errors, we observe that the convergence order of 
the error under the energy norm $\enorm{u-u_h, \bm{p}-\bm{p}_h}$ is 
$O(h^m)$, which is consistent to the theoretical result in Section 
\ref{sec_method}. In addition, for the $L^2$ errors, we can see that 
$\|u - u_h \|_{L^2(\Omega)}$ and $ \| \bm{p} - \bm{p}_h \|
_{L^2(\Omega)}$ converge to zero at the rate $O(h^{m+1})$ and $O(h^m)$, 
respectively, as the mesh is refined. Due to the finite machine 
precision, the convergence order is lower than the expected result 
for the case $m=4$ with the finest mesh.
The pollution effect occurs as the wavenumber $k$ increases, since 
all the errors between the numerical solution and the exact solution 
become larger. 
\begin{table}
  \centering
    \renewcommand\arraystretch{1.3}
    \scalebox{1.}{
    \begin{tabular}{p{0.5cm} | p{3.3cm} | p{1.6cm} | p{1.6cm} |
      p{1.6cm} | p{1.6cm} | p{1cm} }
      \hline\hline
      $m$ & mesh size & $1/5$ & $1/10$ & $1/20$ & $1/40$ & order \\
      \hline
      \multirow{3}{*}{$1$} & $\enorm{(u-u_h, \bm{p}-\bm{p}_h)}$ &
      6.468e-2 & 3.240e-2 & 1.620e-2 & 8.095e-3 & 1.00 \\
      \cline{2-7}
      & $\| u-u_h \|_{L^2(\Omega)}$ & 2.382e-3 & 6.074e-4 & 1.532e-4
      & 3.844e-5 & 2.00 \\
      \cline{2-7}
      & $\| \bm{p} - \bm{p}_h \|_{L^2(\Omega)} $ & 1.980e-2 & 1.008e-2
      & 5.026e-3 & 2.498e-3 & 0.99 \\
      \hline
      \multirow{3}{*}{$2$} & $\enorm{(u-u_h, \bm{p}-\bm{p}_h)}$ &
      1.385e-3 & 3.492e-4 & 8.758e-5 & 2.193e-5 & 2.00 \\
      \cline{2-7}
      & $\| u-u_h \|_{L^2(\Omega)}$ & 1.905e-5 & 2.378e-6 &
      2.968e-7 & 3.707e-8 & 3.00 \\
      \cline{2-7}
      & $\| \bm{p} - \bm{p}_h \|_{L^2(\Omega)} $ & 4.084e-4 & 1.082e-4
      & 2.765e-5 & 6.981e-6 & 1.99 \\
      \hline
      \multirow{3}{*}{$3$} & $\enorm{(u-u_h, \bm{p}-\bm{p}_h)}$ &
      2.085e-5 & 2.624e-6 & 3.301e-7 & 4.145e-8 & 3.00 \\
      \cline{2-7}
      & $\| u-u_h \|_{L^2(\Omega)}$ & 2.407e-7 & 1.514e-8 &
       9.533e-10 & 5.987e-11 & 4.00 \\
      \cline{2-7}
      & $\| \bm{p} - \bm{p}_h \|_{L^2(\Omega)} $ & 5.464e-6 &
      7.259e-7 & 9.509e-8 & 1.227e-8 & 2.99 \\
      \hline
      \multirow{3}{*}{$4$} & $\enorm{(u-u_h, \bm{p}-\bm{p}_h)}$ &
       2.482e-7 & 1.552e-8 & 9.704e-10 & 1.756e-10 & 3.48 \\
      \cline{2-7}
      & $\| u-u_h \|_{L^2(\Omega)}$ & 2.250e-9 & 6.999e-11 & 2.191e-12 
      & 1.232e-12 & 4.72 \\
      \cline{2-7}
      & $\| \bm{p} - \bm{p}_h \|_{L^2(\Omega)} $ & 6.688e-8 & 
      4.229e-9 & 2.659e-10 & 1.330e-10 & 2.99 \\
      \hline
    \end{tabular}}
    \caption{Convergence history for Example 1 with $k=1$.}
    \label{tab_ex1k1}
\end{table}
\begin{table}
  \centering
    \renewcommand\arraystretch{1.3}
    \scalebox{1.}{
    \begin{tabular}{p{0.5cm} | p{3.3cm} | p{1.6cm} | p{1.6cm} |
      p{1.6cm} | p{1.6cm} | p{1cm} }
      \hline\hline
      $m$ & mesh size & $1/5$ & $1/10$ & $1/20$ & $1/40$ & order \\
      \hline
      \multirow{3}{*}{$1$} & $\enorm{(u-u_h, \bm{p}-\bm{p}_h)}$
      & 2.803e-1 & 1.327e-1 & 6.520e-2 & 3.243e-2 & 1.00 \\
      \cline{2-7}
      & $\| u-u_h \|_{L^2(\Omega)}$
      & 2.112e-2 & 5.557e-3 & 1.412e-3 & 3.550e-4 & 1.99 \\
      \cline{2-7}
      & $\| \bm{p} - \bm{p}_h \|_{L^2(\Omega)} $
      & 4.890e-2 & 2.154e-2 & 1.024e-2 & 5.020e-3 & 1.10 \\
      \hline
      \multirow{3}{*}{$2$} & $\enorm{(u-u_h, \bm{p}-\bm{p}_h)}$
      & 1.109e-2 & 2.793e-3 & 7.006e-4 & 1.754e-4 & 2.00 \\
      \cline{2-7}
      & $\| u-u_h \|_{L^2(\Omega)}$
      & 1.628e-4 & 1.937e-5 & 2.386e-6 & 2.970e-7 & 3.00 \\
      \cline{2-7}
      & $\| \bm{p} - \bm{p}_h \|_{L^2(\Omega)} $
      & 1.629e-3 & 4.323e-4 & 1.106e-4 & 2.792e-5 & 1.99 \\
      \hline
      \multirow{3}{*}{$3$} & $\enorm{(u-u_h, \bm{p}-\bm{p}_h)}$
      & 3.340e-4 & 4.199e-5 & 5.281e-6 & 6.632e-7 & 3.00 \\
      \cline{2-7}
      & $\| u-u_h \|_{L^2(\Omega)}$
      & 3.858e-6 & 2.424e-7 & 1.525e-8 & 9.578e-10 & 4.00 \\
      \cline{2-7}
      & $\| \bm{p} - \bm{p}_h \|_{L^2(\Omega)} $
      & 4.395e-5 & 5.814e-6 & 7.609e-7 & 9.812e-8 & 2.99 \\
      \hline
      \multirow{3}{*}{$4$} & $\enorm{(u-u_h, \bm{p}-\bm{p}_h)}$ &
       7.939e-6 & 4.967e-7 & 3.104e-8 & 1.941e-9 & 3.99 \\
      \cline{2-7}
      & $\| u-u_h \|_{L^2(\Omega)}$ & 7.229e-8 & 2.242e-9 &
       6.978e-11 & 2.367e-12 & 4.96 \\
      \cline{2-7}
      & $\| \bm{p} - \bm{p}_h \|_{L^2(\Omega)} $ & 1.067e-6 &
       6.763e-8 & 4.236e-9 & 2.664e-10 & 3.99 \\
      \hline
    \end{tabular}
    }
    \caption{Convergence history for Example 1 with $k=2$.}
    \label{tab_ex1k2}
\end{table}
\begin{table}
  \centering
    \renewcommand\arraystretch{1.3}
    \scalebox{1.}{
    \begin{tabular}{p{0.5cm} | p{3.3cm} | p{1.6cm} | p{1.6cm} |
      p{1.6cm} | p{1.6cm} | p{1cm} }
      \hline\hline
      $m$ & mesh size & $1/5$ & $1/10$ & $1/20$ & $1/40$ & order \\
      \hline
      \multirow{3}{*}{$1$} & $\enorm{(u-u_h, \bm{p}-\bm{p}_h)}$
      & 1.220e+1 & 9.277e+0 & 5.208e+0 & 1.963e+0 & 0.87 \\
      \cline{2-7}
      & $\| u-u_h \|_{L^2(\Omega)}$
      & 7.373e-1 & 5.652e-1 & 3.170e-1 & 1.174e-1 & 0.87 \\
      \cline{2-7}
      & $\| \bm{p} - \bm{p}_h \|_{L^2(\Omega)} $
      & 7.511e-1 & 5.715e-1 & 3.201e-1 & 1.193e-1 & 0.87 \\
      \hline
      \multirow{3}{*}{$2$} & $\enorm{(u-u_h, \bm{p}-\bm{p}_h)}$
      & 2.916e+0 & 3.122e-1 & 4.786e-2 & 1.127e-2 & 2.66 \\
      \cline{2-7}
      & $\| u-u_h \|_{L^2(\Omega)}$
      & 1.753e-1 & 1.586e-2 & 1.048e-3 & 6.796e-5 & 3.76 \\
      \cline{2-7}
      & $\| \bm{p} - \bm{p}_h \|_{L^2(\Omega)} $
      & 1.772e-1 & 1.724e-2 & 2.041e-3 & 4.507e-4 & 2.90 \\
      \hline
      \multirow{3}{*}{$3$} & $\enorm{(u-u_h, \bm{p}-\bm{p}_h)}$
      & 1.066e-1 & 1.084e-2 & 1.353e-3 & 1.698e-4 & 3.10 \\
      \cline{2-7}
      & $\| u-u_h \|_{L^2(\Omega)}$
      & 3.945e-3 & 9.116e-5 & 4.071e-6 & 2.463e-7 & 4.63 \\
      \cline{2-7}
      & $\| \bm{p} - \bm{p}_h \|_{L^2(\Omega)} $
      & 4.905e-3 & 3.859e-4 & 4.888e-5 & 6.281e-6 & 3.23 \\
      \hline
      \multirow{3}{*}{$4$} & $\enorm{(u-u_h, \bm{p}-\bm{p}_h)}$
      & 8.138e-3 & 5.083e-4 & 3.180e-5 & 1.987e-6 & 4.00 \\
      \cline{2-7}
      & $\| u - u_h \|_{L^2(\Omega)}$
      & 7.508e-5 & 1.828e-6 & 5.472e-8 & 1.688e-9 & 5.14 \\
      \cline{2-7}
      & $\| \bm{p} - \bm{p}_h \|_{L^2(\Omega)}$
      & 2.777e-4 & 1.753e-5 & 1.105e-6 & 6.919e-8 & 3.99 \\
      \hline
    \end{tabular}
    }
    \caption{Convergence history for Example 1 with $k=8$.}
    \label{tab_ex1k8}
\end{table}

\noindent \textbf{Example 2.} For the second example, we consider a 2d
example defined on $\Omega = (-0.5,0.5)^2$
\cite{Feng2009discontinuous},
\begin{eqnarray*}
  \left\{ \begin{array}{ll}
    -\Delta u - k^2 u &= f := \frac{\sin(kr)}{r}, \qquad
    \text{in } \Omega, \\
    \npar{u} + \ui k u &= g, \qquad \text{on } \partial \Omega.
  \end{array}
  \right.
\end{eqnarray*}
The analytical solution can be written as
\begin{displaymath}
  u = \frac{\cos(kr)}{k} - \frac{\cos k + \ui \sin k}{k (J_0(k) + \ui 
  J_1(k))} J_0(kr),
\end{displaymath}                                                     
in the polar coordinates $(r, \theta)$, where $J_v(z)$ are Bessel 
functions of the first kind. 

First, we test the convergence order for the case $k=1$. 
We set the initial mesh size to be $h = 1/5$ and uniformly refine
the mesh for three times to solve this problem. 
The numerical errors are shown in Tab.~\ref{tab_ex2k1} with the degree
of approximation spaces $m = 1,2,3$. We observe that the numerical
error under the energy norm tends to zero at the speed $O(h^m)$ as the
mesh size approachs to zero, and the convergence order of $L^2$ errors
are $O(h^{m+1})$ for the variable $u$ and $O(h^m)$ for the varible
$\bm{p}$. We note that all these results are still consistent with the 
theoretical error estimates. Fig.~\ref{fig_surface} exhibits the 
surface plots of the exact solution and the numerical solution for
$k=100$.

Next, we numerically examine the changes of the error under the energy
norm when the wavenumber $k$ and the mesh size $h$ are correlated. 
We use piecewise linear spaces to approximate the variables $u$ and 
$\bm{p}$, so that the error estimate in Theorem \ref{th_estimate} 
suggests that
\begin{displaymath}
  \enorm{(u-u_h, \bm{p}-\bm{p}_h)} \leq C k^2 h(1+h+k^2h^2) (1+ k^2 
  h^2)^{\frac{1}{2}} (\| u \|_{H^{2}(\Omega)} +
  \| \bm{p} \|_{H^{2}(\Omega)}).      
\end{displaymath}
In Fig.~\ref{fig_k2h}, we plot the relative energy error of the
discontinuous least squares method for $k$ and $h$ determined by $k^2h
= 1$. We see that the error gradually decreases and tends to be
invariant when $k$ becomes large, which verifies our $k$-explicit
error estimates.
\begin{table}
  \centering
    \renewcommand\arraystretch{1.3}
    \scalebox{1.}{
    \begin{tabular}{p{0.5cm} | p{3.3cm} | p{1.6cm} | p{1.6cm} |
      p{1.6cm} | p{1.6cm} | p{1cm} }
      \hline\hline
      $m$ & mesh size & $1/5$ & $1/10$ & $1/20$ & $1/40$ & order \\
      \hline
      \multirow{3}{*}{$1$} & $\enorm{(u-u_h, \bm{p}-\bm{p}_h)}$
      & 3.386e-2 &  1.692e-2 & 8.466e-3 & 4.234e-3 & 1.00 \\
      \cline{2-7}
      & $\| u-u_h \|_{L^2(\Omega)}$
      & 1.848e-3 & 4.664e-4 & 1.170e-4 & 2.929e-5 & 1.99 \\
      \cline{2-7}
      & $\| \bm{p} - \bm{p}_h \|_{L^2(\Omega)} $
      & 2.430e-3 & 1.763e-3 & 9.741e-4 & 4.993e-4 &  0.76 \\
      \hline
      \multirow{3}{*}{$2$} & $\enorm{(u-u_h, \bm{p}-\bm{p}_h)}$
      & 1.135e-3 & 2.841e-4 & 7.106e-5 & 1.777e-5 & 2.00 \\
      \cline{2-7}
      & $\| u-u_h \|_{L^2(\Omega)}$
      & 6.182e-6 & 7.257e-7 & 8.910e-8 & 1.107e-8 & 3.03 \\
      \cline{2-7}
      & $\| \bm{p} - \bm{p}_h \|_{L^2(\Omega)} $
      & 9.629e-5 & 2.723e-5 & 7.185e-6 & 1.839e-6 & 1.90 \\
      \hline
      \multirow{3}{*}{$3$} & $\enorm{(u-u_h, \bm{p}-\bm{p}_h)}$
      & 1.543e-5 & 1.938e-6 & 2.428e-7 & 3.039e-8 & 3.00 \\
      \cline{2-7}
      & $\| u-u_h \|_{L^2(\Omega)}$
      & 2.622e-7 & 1.633e-8 & 1.022e-9 & 6.403e-11 & 4.00 \\
      \cline{2-7}
      & $\| \bm{p} - \bm{p}_h \|_{L^2(\Omega)} $
      & 1.778e-6 & 2.772e-7 & 3.738e-8 & 4.845e-9 & 2.83\\
      \hline
    \end{tabular}
    }
    \caption{Convergence history for Example 2 with $k=1$.}
    \label{tab_ex2k1}
\end{table}

\begin{figure}
  \centering
  \includegraphics[width=0.4\textwidth]{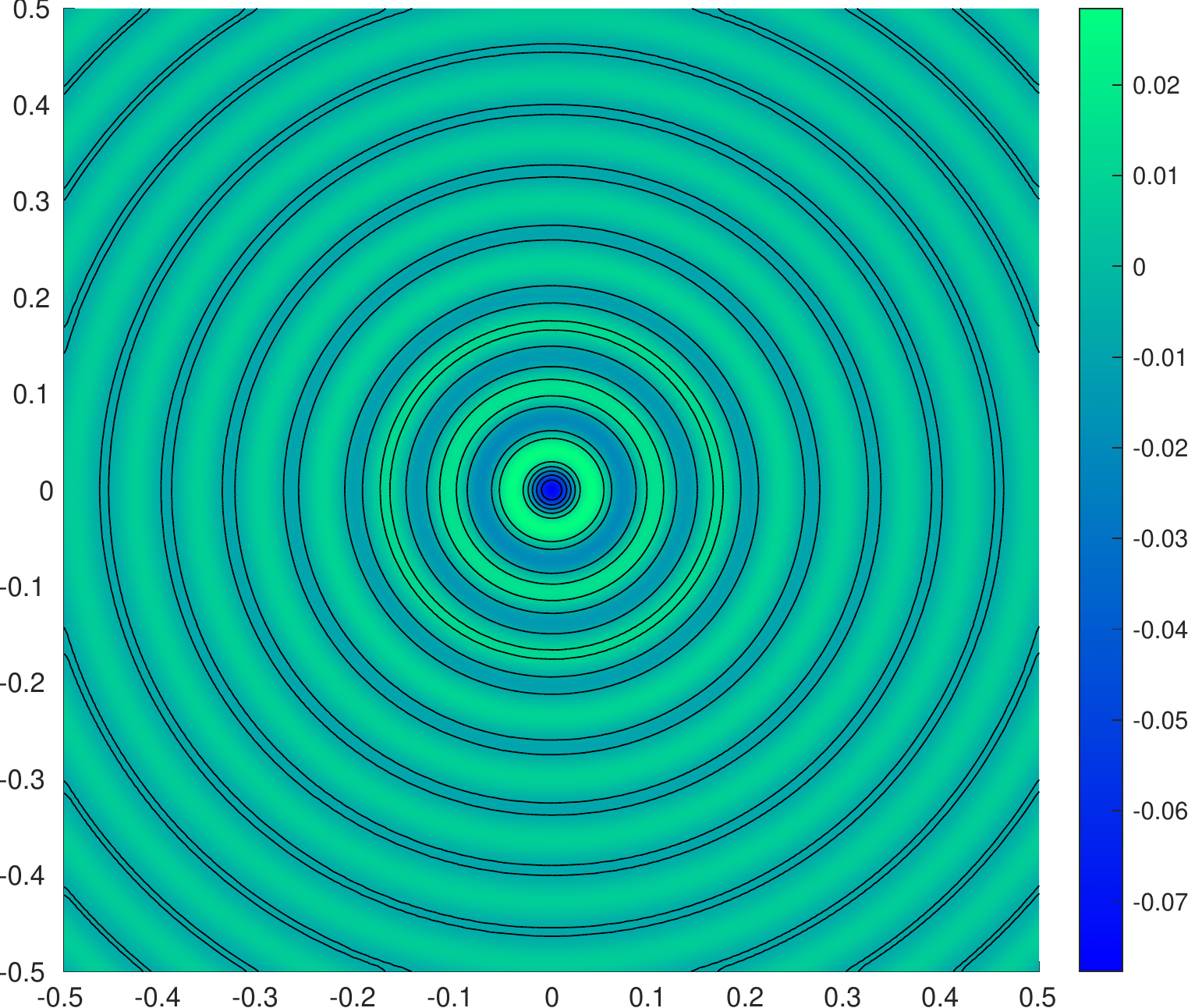}
  \hspace{25pt}
  \includegraphics[width=0.4\textwidth]{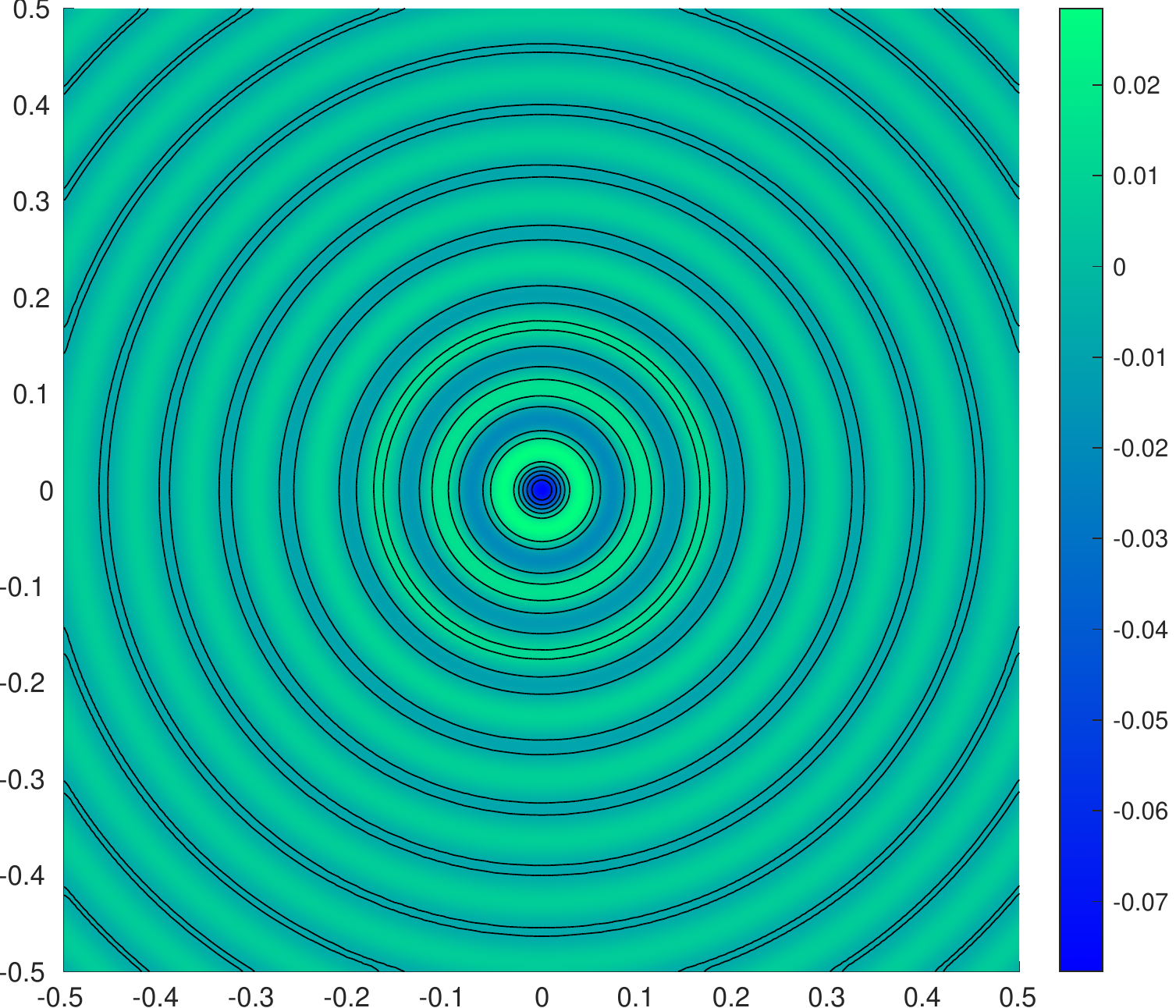}
  \caption{Surface plots for the exact solution of Example 2 
  (left) and the numerical solution with $k=100$ and $m=3$ (right).
  The number of elements is 139264.}
  \label{fig_surface}
\end{figure}

\begin{figure}
  \centering
  \includegraphics[width = 0.6\textwidth]{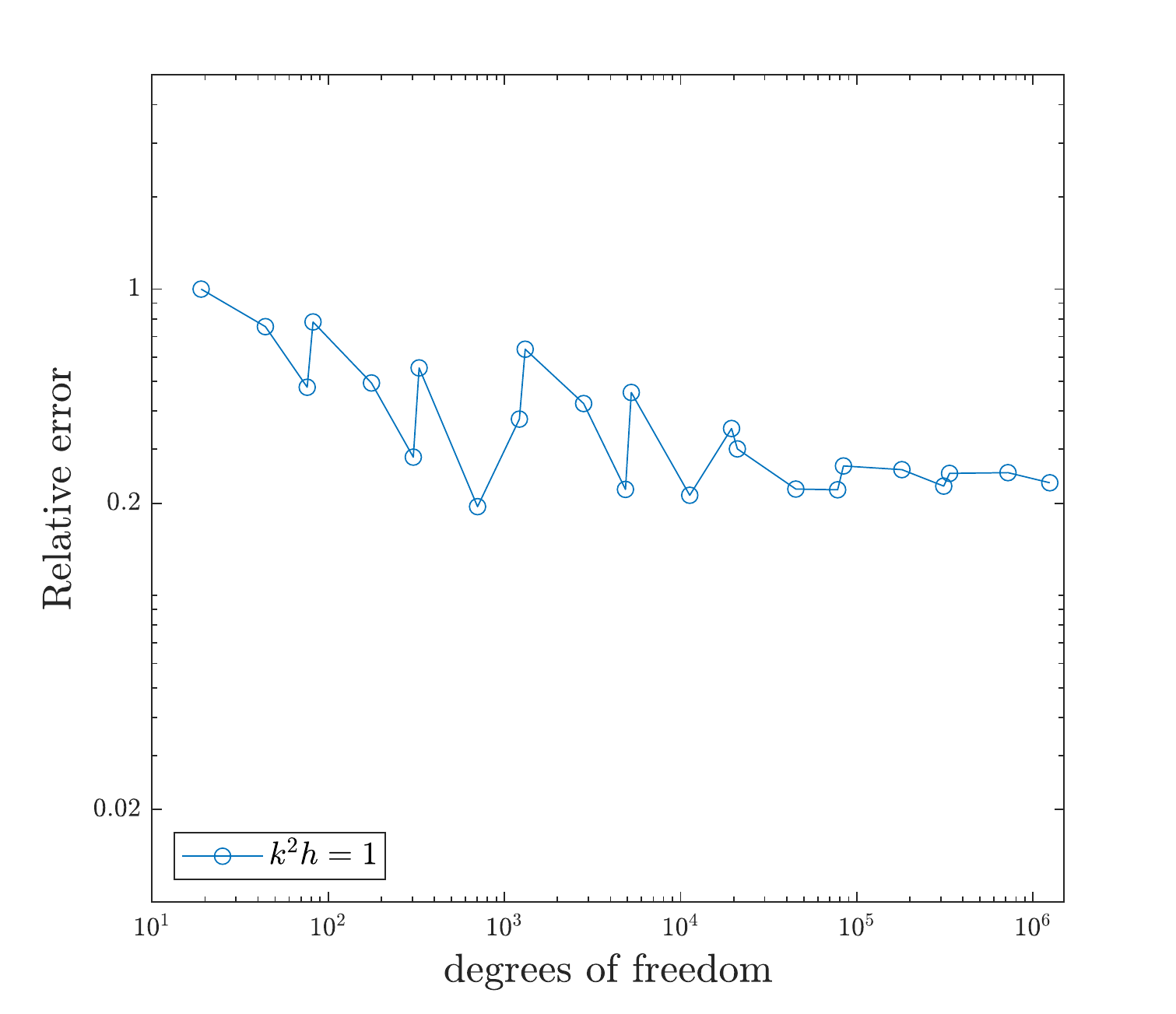}
  \caption{Relative error of Example 2 with $k^2 h = 1$.}
  \label{fig_k2h}
\end{figure}

\noindent \textbf{Example 3.} In this example, we solve a             
three-dimensional problem defined in the cube $\Omega = (-1,1)^3$. 
The analytical solution is selected as
\begin{displaymath}
  u(x,y,z) = \mr{e}^{\ui k (x \sin \theta \cos \phi + y \sin \theta
  \sin \phi + z \cos \theta)},
\end{displaymath}
where the parameter $\theta$ and $\phi$ are set to be $\frac{\pi}{4}$
and $\frac{\pi}{5}$ respectively. We solve this test problem on a 
series of tetrahedral meshes with the resolution $h = 1/4$, $1/8$,
$1/16$, and $1/32$, see Fig.~\ref{fig_partition}. We use the
approximation spaces $\bmr{V}_h^m$ and $\bmr{\Sigma}_h^m$ to
approximate $u$ and $\bm{p}$, respectively. The convergence
histories for $k=1$ are displayed in Tab.~\ref{tab_ex3k1}. We obseve 
that the convergence order under the energy norm $\enorm{\cdot}$ is
still the optimal order $O(h^m)$, and the $L^2$ errors for $u$ and
$\bm{p}$ are still $O(h^{m})$ and $O(h^{m+1})$, respectively. We note 
that all numerical convergence orders are consistent with the 
theoretical error estimate as before.
\begin{table}
  \centering
    \renewcommand\arraystretch{1.3}
    \scalebox{1.}{
    \begin{tabular}{p{0.5cm} | p{3.3cm} | p{1.6cm} | p{1.6cm} |
      p{1.6cm} | p{1.6cm} | p{1cm} }
      \hline\hline
      $m$ & mesh size & $1/4$ & $1/8$ & $1/16$ & $1/32$ & order \\
      \hline
      \multirow{3}{*}{$1$} & $\enorm{(u-u_h, \bm{p}-\bm{p}_h)}$
      & 1.940e-1 & 9.754e-2 & 4.898e-2 & 2.458e-2 & 0.99 \\
      \cline{2-7}
      & $\| u-u_h \|_{L^2(\Omega)}$
      & 1.591e-2 & 4.333e-3 & 1.117e-3 & 2.829e-4 & 1.94 \\
      \cline{2-7}
      & $\| \bm{p} - \bm{p}_h \|_{L^2(\Omega)} $
      & 6.325e-2 & 3.538e-2 & 1.875e-2 & 9.595e-3 &  0.90 \\
      \hline
      \multirow{3}{*}{$2$} & $\enorm{(u-u_h, \bm{p}-\bm{p}_h)}$
      & 1.218e-2 & 3.181e-3 & 8.055e-4 & 2.030e-4 & 1.96 \\
      \cline{2-7}
      & $\| u-u_h \|_{L^2(\Omega)}$
      & 4.952e-4 & 6.136e-5 & 7.612e-6 & 9.542e-7 & 3.00 \\
      \cline{2-7}
      & $\| \bm{p} - \bm{p}_h \|_{L^2(\Omega)} $
      & 3.781e-3 & 1.194e-3 & 3.209e-4 & 8.287e-5 & 1.83 \\
      \hline
      \multirow{3}{*}{$3$} & $\enorm{(u-u_h, \bm{p}-\bm{p}_h)}$
      & 5.628e-4 & 7.438e-5 & 9.458e-6 & 1.198e-6 & 2.95 \\
      \cline{2-7}
      & $\| u-u_h \|_{L^2(\Omega)}$
      & 1.770e-5 & 1.153e-6 & 7.299e-8 & 4.618e-9 & 3.96 \\
      \cline{2-7}
      & $\| \bm{p} - \bm{p}_h \|_{L^2(\Omega)} $
      & 1.741e-4 & 2.808e-5 & 3.818e-6 & 4.991e-7 & 2.81\\
      \hline
    \end{tabular}
    }
    \caption{Convergence history for Example 3 with $k=1$.}
    \label{tab_ex3k1}
\end{table}

\noindent \textbf{Example 4.} In this test, we apply the proposed
method to a problem with low regularity near the origin. The domain
$\Omega$ is selected to be an L-shaped domain $\Omega = (-1,1) \backslash [0,1)
\times (-1,0]$. We set $f=0$ and choose the exact solution, in polar
coordinates $(r, \theta)$, to be
\begin{equation*}
  u(x,y) = J_{\alpha}(kr) \cos(\alpha \theta).
\end{equation*}
This exact solution belongs to the space $H^{\alpha + 1-\epsilon}(\Omega)$.
We select the parameter $\alpha = 2/3$ and set the initial mesh size
to be $h = 1/4$. We uniformly refine the mesh for three times to solve 
this problem for $k=1$. Tab.~\ref{tab_ex4k1} shows the convergence rate of
$\|u-u_h\|_{L^2(\Omega)}$ and $\|\bm{p} - \bm{p}_h \|_{L^2(\Omega)}$
with $m=1,2,3$. The convergence rate of $\|\bm{p} - \bm{p}_h \|
_{L^2(\Omega)}$ is about $0.67$, which is in agreement with with the 
regularity of the exact solution and error estimates. For the error 
$\|u-u_h\|_{L^2(\Omega)}$, we note that the convergence rate is lower
than its regularity exponent, and seems to decrease when $m$
increases. 
\begin{table}
  \centering
    \renewcommand\arraystretch{1.3}
    \scalebox{1.}{
    \begin{tabular}{p{0.5cm} | p{3.3cm} | p{1.6cm} | p{1.6cm} |
      p{1.6cm} | p{1.6cm} | p{1cm} }
      \hline\hline
      $m$ & mesh size & $1/4$ & $1/8$ & $1/16$ & $1/32$ & order \\
      \hline
      \multirow{2}{*}{$1$} & $\| u-u_h \|_{L^2(\Omega)}$
      & 1.019e-2 & 3.307e-3 & 1.109e-3 & 3.872e-4 & 1.57 \\
      \cline{2-7}
      & $\| \bm{p} - \bm{p}_h \|_{L^2(\Omega)} $
      & 8.031e-2 & 4.900e-2 & 3.061e-2 & 1.920e-2 & 0.67 \\
      \hline
      \multirow{2}{*}{$2$} & $\| u-u_h \|_{L^2(\Omega)}$
      & 1.292e-3 & 4.483e-4 & 1.641e-4 & 6.226e-5 & 1.45 \\
      \cline{2-7}
      & $\| \bm{p} - \bm{p}_h \|_{L^2(\Omega)} $
      & 4.023e-2 & 2.619e-2 & 1.652e-2 & 1.041e-2 & 0.66 \\
      \hline
      \multirow{2}{*}{$3$} & $\| u-u_h \|_{L^2(\Omega)}$
      & 7.098e-4 & 2.677e-4 & 1.032e-4 & 4.309e-5 & 1.37 \\
      \cline{2-7}
      & $\| \bm{p} - \bm{p}_h \|_{L^2(\Omega)} $
      & 2.073e-2 & 1.706e-2 & 1.075e-2 & 6.776e-3 & 0.66\\
      \hline
    \end{tabular}
    }
    \caption{Convergence history for Example 4 with $k=1$.}
    \label{tab_ex4k1}
  \end{table}

\noindent \textbf{Example 5.} In this example, we consider a          
circumferentially harmonic radiation from a rigid infinite circular   
cylinder of radius $a$ \cite{Harari1992galerkin}. The exact solution 
is chosen by 
\begin{equation}
  u(x,y) = \frac{H^{(1)}_n(kr) \cos n\theta}{H^{(1)}_n (ka)},
\end{equation}
where $H^{(1)}_n$ is the Hankel function of the first kind of order 
$n$. The domain is set to be a circular ring $\Omega = B(0,2a) 
\slash B(0,a)$. We apply the Dirichlet boundary condition on $\partial 
B(0,a)$ and the Robin boundary condition on $\partial B(0,2a)$. In our
numerical simulation, we compute the fifth circumferential mode ($n =
4$) and choose $k=\pi$, $a=1$. We use the discontinuous piecewise linear 
approximation spaces $\bmr{V}_h^1 \times \bmr{\Sigma}_h^1$ for this 
example. We use the polygon approximation to the domain $\Omega$, and 
then triangulate it into a shape-regular mesh, see Fig.~\ref{fig_ex5mesh} 
. In Fig.~\ref{fig_exhan}, we show the contours of the real part of the 
numerical solution and the exact solution, respectively. We observe 
that the least squres discontinuous finite element solution recovers 
the essential features of the exact solution.
\begin{figure}
  \centering
  \includegraphics[width=0.5\textwidth]{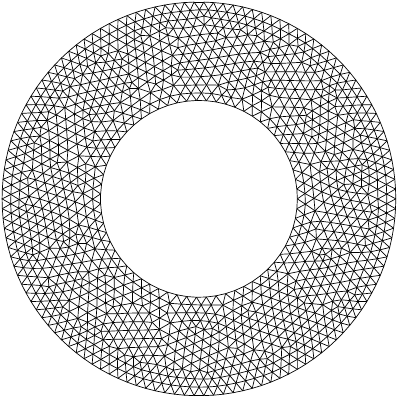}
  \caption{The mesh used in Example 5 with 8908 elements.}
  \label{fig_ex5mesh}
\end{figure}
\begin{figure}
  \centering
  \includegraphics[width=0.4\textwidth]{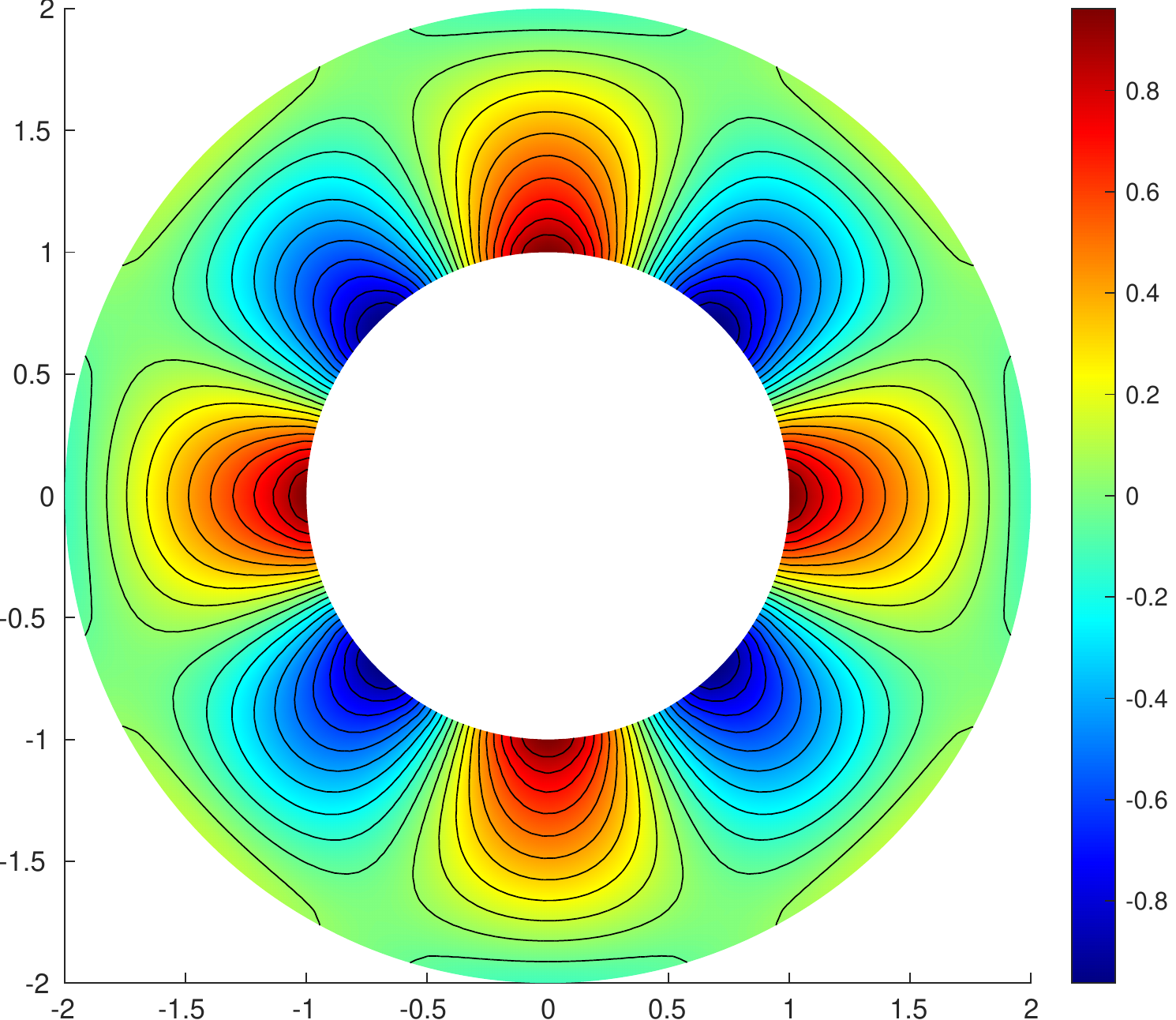}
  \hspace{15pt}
  \includegraphics[width=0.4\textwidth]{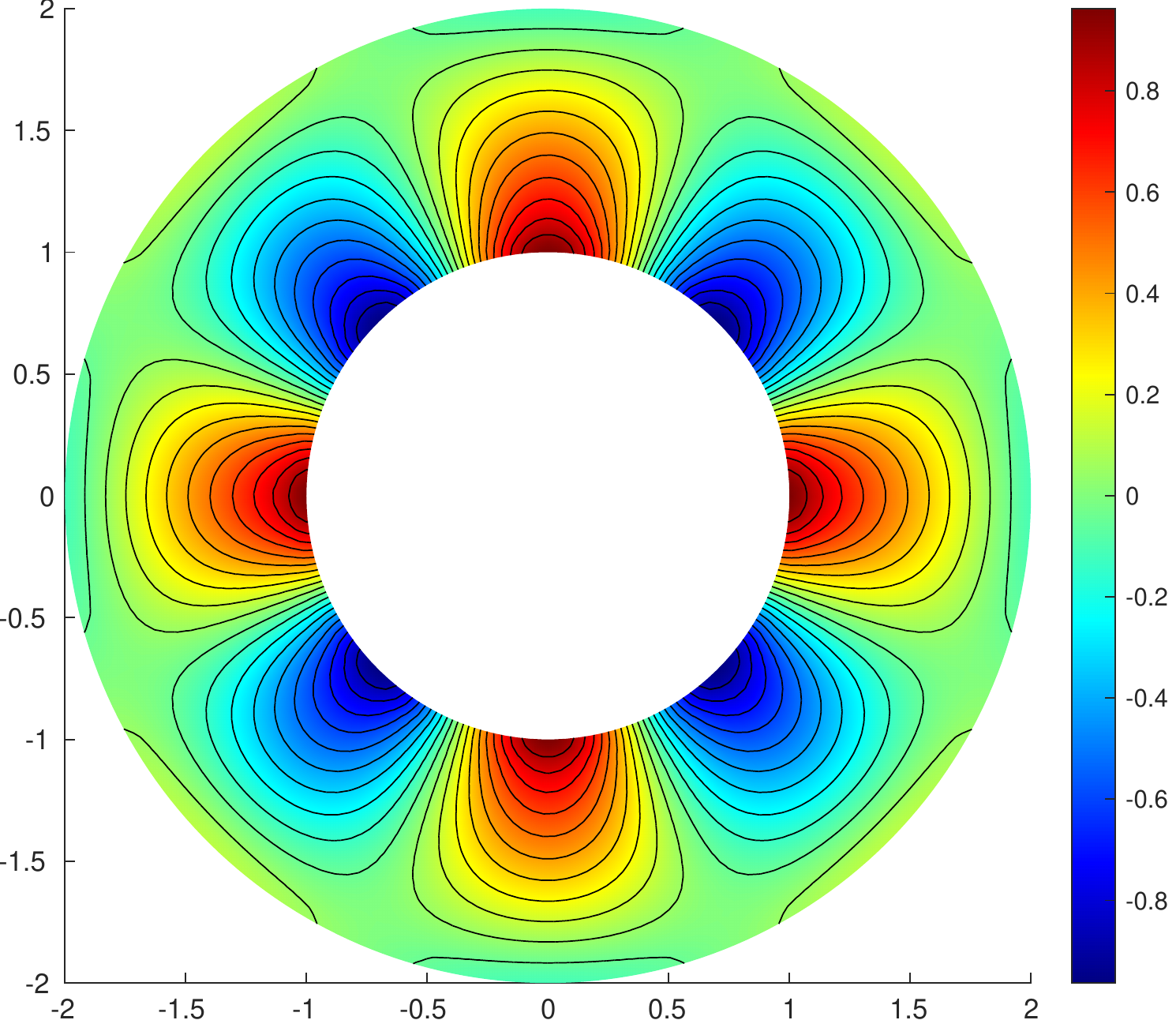}
  \caption{The real part of the numerical solution (left) and the exact
  solution (right) with $a=1$, $k=\pi$ and $m=1$.}
  \label{fig_exhan}
\end{figure}

\noindent \textbf{Example 6.} In this example, we test the performance 
of our adaptive algorithm proposed in Section \ref{sec_method}.
We solve the low-regularity problem defined in Example 4 with $\alpha
= 2/3$. For the adaptive algorithm, we choose the parameter 
$\lambda = 0.45$ and we use the longest-edge bisection algorithm to 
refine the mesh. We use approximation spaces with $m=1$ to solve the 
problem. In Fig.~\ref{fig_ex6}, we compare the original mesh (left) 
with the mesh after 5 adaptive refinement steps (right). The mesh is 
refined remarkably around the corner $(0, 0)$, where the exact solution        
contains a singularity. The convergence history under $L^2$ norms is 
displayed in Fig.~\ref{fig_adaptive}. From Fig.~\ref{fig_adaptive}, we
see that the convergence orders of $\| u- u_h \|_{L^2(\Omega)}$ and 
$\| \bm{p}-\bm{p}_h \|_{L^2(\Omega)}$ are $O(N^{-1})$ and $O(N^{-1/2})$ 
, respectively, where $N$ is the number of degrees of freedom. These
results match the convergence rates for smooth cases in Example
1 and Example 2. The convergence rates are better than that in
Tab.~\ref{tab_ex4k1}, where the $L^2$ errors tend to zero at the speed        
$O(N^{-1.57/2})$ and $O(N^{-0.67/2})$ for the variables $u$ and  
$\bm{p}$, respectively.
\begin{figure}
  \centering
  \includegraphics[width=0.4\textwidth]{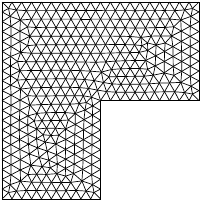}
  \hspace{15pt}
  \includegraphics[width=0.4\textwidth]{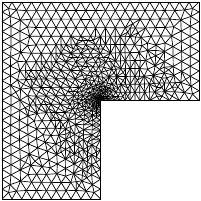}
  \caption{The initial mesh (left) / The mesh after 5 adaptive
  refinement steps (right)}
  \label{fig_ex6}
\end{figure}

\begin{figure}
  \centering    
  \includegraphics[width=0.4\textwidth]{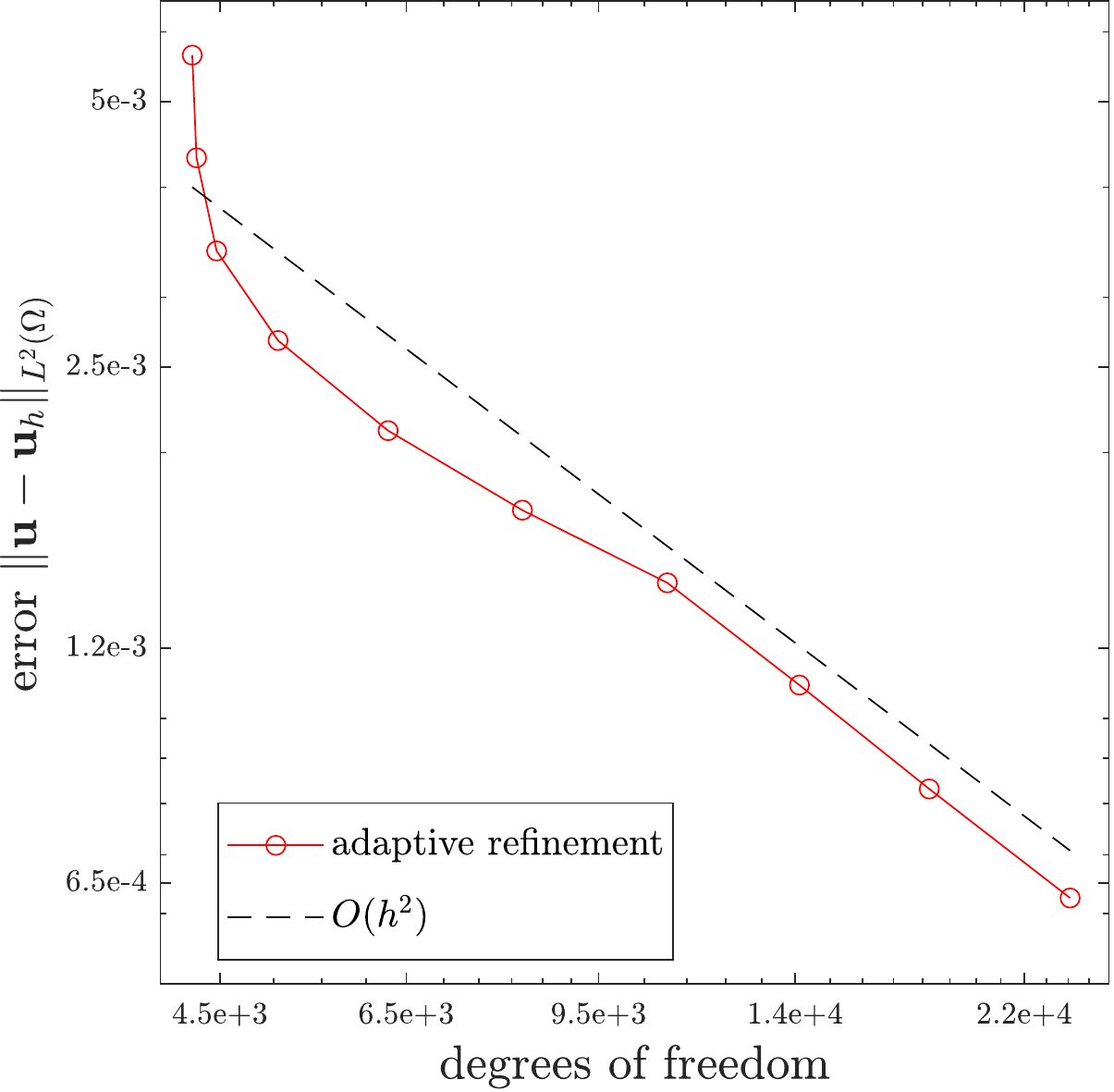}              
  \hspace{15pt}
  \includegraphics[width=0.4\textwidth]{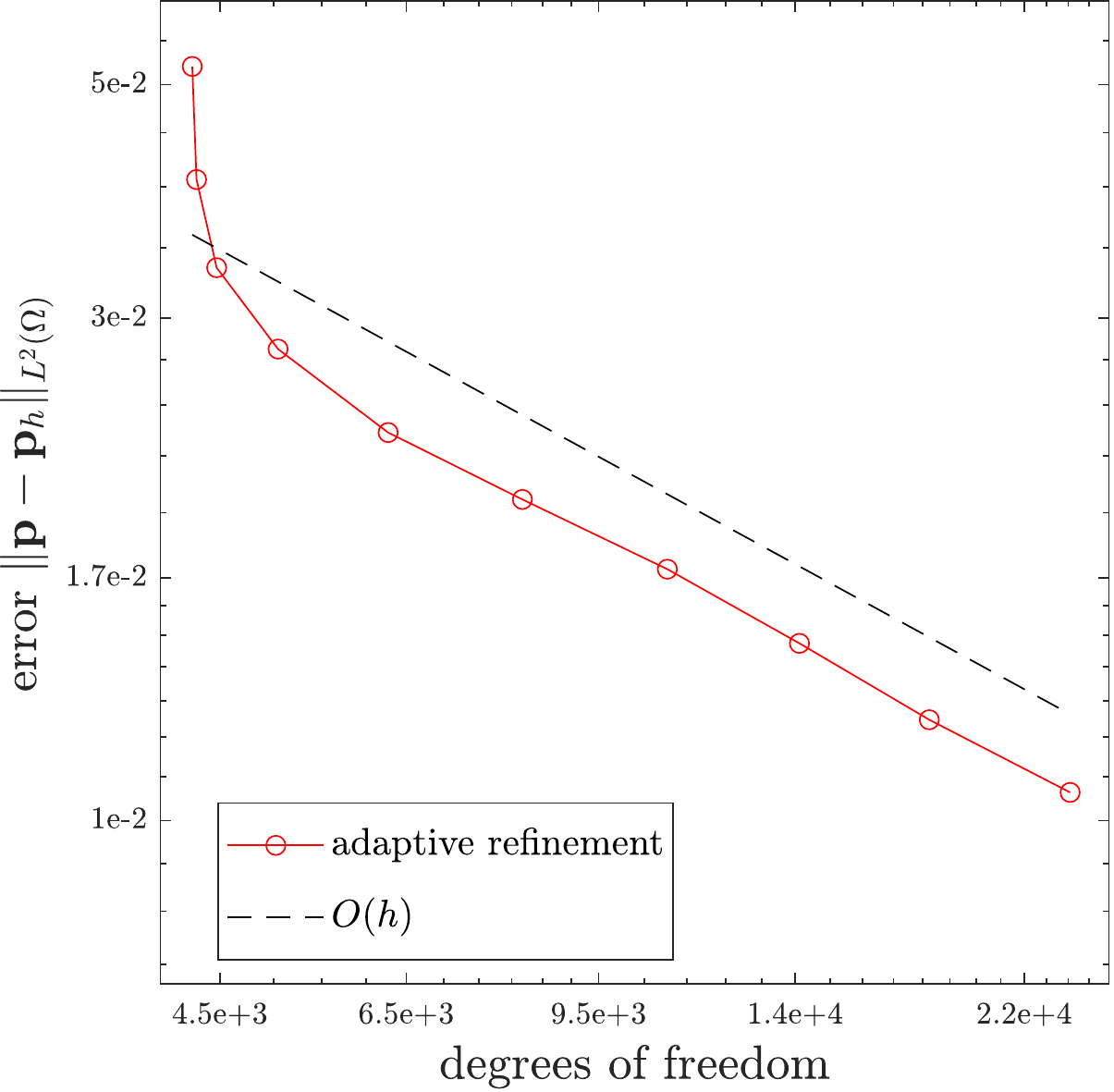}
  \caption{Convergence history for Example 6.}
  \label{fig_adaptive}
\end{figure}

%% file: conclusion.tex
\section{Conclusions}
\label{sec_conclusions}
We proposed a discontinuous least squares finite element method for
the Helmholtz equation. We designed an $L^2$ norm least squares 
functional with the weak imposition of the continuity across the 
interior faces, and minimized it over the discontinuous approximation
space $\bmr{V}_h^m \times \bmr{\Sigma}_h^m$. We established the
$k$-explicit error estimates for our method. The convergence rates
were derived to be optimal under the energy norm and suboptimal under
the $L^2$ norm for a fixed wavenumber $k$. Particularly, it was proved 
that our method is stable without any constraint on the mesh size.
Numerical results in both two and three dimensions verified the 
accuracy of our method. 

\section*{Acknowledgements}
This research was supported by the National Natural Science Foundation in China
(No. 11971041).